\documentclass[11pt,reqno,oneside]{amsart}
\usepackage[letterpaper]{geometry}
\usepackage[rgb]{xcolor}
\usepackage[colorlinks, allcolors=blue]{hyperref}
\usepackage{amsmath,amsthm,amssymb,tikz,tikz-cd, ifthen,enumerate}

\allowdisplaybreaks 

\usepackage{cleveref} 
\newtheorem{thm}{Theorem}
\newtheorem{cor}[thm]{Corollary}
\newtheorem{lem}[thm]{Lemma}
\newtheorem{prop}[thm]{Proposition}

\theoremstyle{definition}
\newtheorem{defn}[thm]{Definition}
\theoremstyle{remark}

\newtheorem*{remark*}{Remark}
\newcommand\<{\begin{equation}} \renewcommand\>{\end{equation}}

\usepackage[labelfont=sc, justification=centering, textformat=period]{caption}
\makeatletter \g@addto@macro\@floatboxreset\centering \makeatother 

\newcommand\abs[1]{\left\vert#1\right\vert}
\newcommand\norm[1]{\left\Vert#1\right\Vert}
\newcommand\floor[1]{{\left\lfloor#1\right\rfloor}}
\newcommand\ceil[1]{{\left\lceil#1\right\rceil}}
\let\tilde\widetilde
\renewcommand\bar[1]{\,\overline{\!#1\!}\,}
\renewcommand\dots{...} \renewcommand\ldots{...}
\newcommand\C{{\mathbb C}}
\newcommand\D{{\mathbb D}}
\newcommand\N{{\mathbb N}}
\newcommand\R{{\mathbb R}}
\newcommand\Sb{{\mathbb S}}
\newcommand\Z{{\mathbb Z}}
\newcommand\Fc{{\mathcal F}}
\newcommand\Tc{{\mathcal T}}
\newcommand\Piecewise[2][]{\left\{\begin{array}{ll} #2 \ifthenelse{\equal{#1}{}}{}{\\ #1 & \text{otherwise}}\end{array}\right.}

\renewcommand\P{{\bar P}} \newcommand\Q{{\bar Q}} \newcommand\A{{\bar A}}
\renewcommand\i{{\rm i}}
\providecommand\CylI[2]{{ I_{\raisebox{-1pt}{\scriptsize\!$#1$}}^{\raisebox{1pt}{\scriptsize$#2$}} }}

\newcounter{commentcounter}\newcommand\COMMENT[2][red]{\stepcounter{commentcounter}\rlap{\smash{$^{\fcolorbox{#1}{#1!15}{\scriptsize\ifthenelse{\equal{#1}{green}}{\color{green!75!black}}{\color{#1}}\!\!{\bf\thecommentcounter}\!\!}}$}}\marginpar{\!\!\parbox{2.8cm}{\raggedright\small \ifthenelse{\equal{#1}{green}}{\color{green!67!black}}{\color{#1}} \textbf{\thecommentcounter.}\,#2}}} \usepackage{silence} 

\begin{document}

\title[Rigidity of topological entropy]{Rigidity of topological entropy of boundary maps associated to Fuchsian groups}
\author{Adam Abrams}
\address{Faculty of Pure and Applied Mathematics, Wroc\l{}aw University of Science and Technology, Wroc\l{}aw, 50370, Poland}
\email{the.adam.abrams@gmail.com}
\author{Svetlana Katok}
\address{Department of Mathematics, The Pennsylvania State University, University Park, PA 16802, USA}
\email{sxk37@psu.edu}
\author{Ilie Ugarcovici}
\address{Department of Mathematical Sciences, DePaul University, Chicago, IL 60614, USA}
\email{iugarcov@depaul.edu}

\thanks{The second author was partially supported by NSF grant DMS 1602409.}

\dedicatory{Dedicated to the memory of Anatole Katok}

\begin{abstract}
Given a closed, orientable surface of constant negative curvature and genus $g \ge 2$, we study a family of generalized Bowen--Series boundary maps and prove the following rigidity result: in this family the topological entropy is constant and depends only on the genus of the surface. We give an explicit formula for this entropy and show that the value of the topological entropy also stays constant in the Teichm\"uller space of the surface. The proofs use conjugation to maps of constant slope.
\end{abstract}

\maketitle

\section{Introduction}\label{sec intro}

The notion of topological entropy was introduced by Adler, Konheim, and McAndrew in~\cite{AKM}. Their definition used covers and applied to compact Hausdorff spaces; Dinaburg~\cite{D70} and Bowen~\cite{B71} gave definitions involving distance functions and separated sets, which are often more suitable for calculations. While these formulations of topological entropy were originally intended for continuous maps acting on compact spaces, Bowen's definition can actually be applied to piecewise continuous, piecewise monotone maps on an interval, as explained in~\cite{MZ}. The theory naturally extends to maps of the circle, where piecewise monotonicity is understood to mean local monotonicity or, equivalently, having a piecewise monotone lift to $\R$.

In~\cite{Parry66}, following his seminal work~\cite{Parry64} on Markov maps, Parry showed that a piecewise monotone, (strongly) transitive interval map with positive topological entropy is conjugate to a constant slope map. In~\cite{MTh}, Milnor and Thurston used kneading theory to prove a semi-conjugacy result for continuous, piecewise monotone, but not necessarily transitive, interval maps.
In~\cite{AM}, following~\cite{ALM}, Alsed\`a and Misiurewicz give a simpler proof that also generalizes to piecewise continuous, piecewise monotone interval maps.

In this paper we apply the results of~\cite{Parry66,AM} to a multi-parameter family of piecewise continuous, piecewise monotone maps of the circle, the so-called ``boundary maps'' for surfaces of constant negative curvature, as in \cite{KU17}. Some particular maps in this family---including those considered by Bowen and Series \cite{BS79} and further studied by Adler and Flatto \cite{AF91}---are Markov, and the topological entropy can be calculated as the logarithm of the maximal eigenvalue of a transition matrix \cite[Theorem 7.13]{W} in these cases. However, not all maps in our family admit a Markov partition, and yet we prove the following rigidity result: in this family, the topological entropy is constant and depends only on the genus of the surface. Therefore, the topological entropy in these non-Markov cases is the same logarithmic expression. We also show that topological entropy stays constant in the Teichm\"uller space of the surface.

Let $\Gamma$ be a finitely generated cocompact Fuchsian group of the first kind acting freely on the unit disc $\D=\{\, z \in \C : \abs z < 1 \,\}$ endowed with hyperbolic metric $\tfrac{2 \abs{dz} }{1-{\abs z}^2}$ such that $S=\Gamma\backslash\D$ is a surface of genus $g\ge 2$.

A classical (Ford) fundamental domain for $\Gamma$ is a $4g$-sided regular polygon centered at the origin. In~\cite{AF91}, Adler and Flatto used another fundamental domain---an $(8g-4)$-sided polygon $\Fc$---that was much more convenient for their purposes. Its sides are geodesic segments which satisfy the {\em extension condition}: the geodesic extensions of these segments never intersect the interior of the tiling sets $\gamma\Fc$, $\gamma\in \Gamma$.

\begin{figure}[hbt]
    \includegraphics[width=0.67\textwidth]{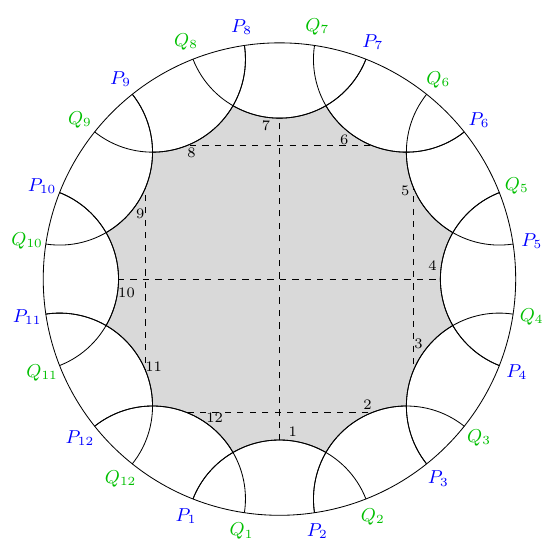}
    \caption{Fundamental polygon $\Fc$ for genus $g=2$}
    \label{fig polygon}
\end{figure}
We denote the endpoints of the oriented infinite geodesic that extends side $k$ to the circle at infinity~$\partial \D$ by $P_k$ and $Q_{k+1}$, where $1\le k\le 8g-4$ is considered mod $8g-4$ throughout this paper (see \Cref{fig polygon}). The counter-clockwise order of endpoints on $\partial\D$ is the following:
\[ P_1, Q_1, P_2, Q_2, \ldots, Q_{8g-4}. \]
The identification of the sides of $\Fc$ is given by the side pairing rule
\[ \sigma(k) := \left\{ \begin{array}{ll}
    4g-k \bmod (8g-4) & \text{ if $k$ is odd} \\
    2-k \bmod (8g-4) & \text{ if $k$ is even.}
\end{array} \right. \]
The generators $T_k$ of $\Gamma$ associated to this fundamental domain are M\"obius transformations satisfying the following properties: denoting $\rho(k) = \sigma(k)+1$ and with $V_k$ as the vertex of~$\Fc$ where sides $k\!-\!1$ and $k$ meet,
\begin{align*}
    T_k(V_k) &= V_{\rho(k)}, &
    T_{\sigma(k)}T_k &= \mathrm{Id}, &
    T_{\rho^3(k)}T_{\rho^2(k)}T_{\rho(k)}T_k &= \mathrm{Id}.
\end{align*}

\begin{remark*}
As functions on $\overline\D \subset \C$, the generators $T_k$ are M\"obius transformations, but restricted to the boundary $\Sb$ they are real functions of the arguments (but \emph{not} fractional linear transformations of the arguments). To simplify notation we will use ``$T_k$'' in both situations: $T_k(z)$ with $z \in \partial\D$ for complex (multiplicative) notation and $T_k(x):=\arg(T_k(e^{\i x}))$ with $x \in \Sb= \R/{2\pi\Z}$ for real (additive) notation. See the left of \Cref{fig T and S} for a plot of $y=T_k(x)$ with~$x,y \in [-\pi,\pi]$.
\end{remark*}

Notice that in general the polygon $\Fc$ need not be regular. In fact, one of the definitions of Teichm\"uller space, used in~\cite{AKU-Flexibility}, is the space of all marked canonical hyperbolic $(8g-4)$-gons in the unit disk~$\D$ (up to an isometry of $\D$) such that side~$k$ and side~$\sigma(k)$ have equal length and the internal angles at vertices~$V_k$ and~$V_{\sigma(k)+1}$ sum to~$\pi$. (The topology on the space of polygons is as follows: ${\mathcal P}_n \to {\mathcal P}$ if and only if the lengths of all sides converge and the measures of all angles converge.)

If $\Fc$ is regular, it is the Ford fundamental domain, i.e., the geodesic from $P_k$ to $Q_{k+1}$ (which we denote as just $P_k Q_{k+1}$) is the isometric circle for $T_k$, and $T_k(P_k Q_{k+1}) = Q_{\sigma(k)+1}P_{\sigma(k)}$ is the isometric circle for $T_{\sigma(k)}$ so that the inside of the former isometric circle is mapped to the outside of the latter, and all internal angles of~$\Fc$ are equal to~$\pi/2$. See~\cite{AK19} for more details and \Cref{sec generators} for additional properties of the generators~$T_k$.

The object of our study is the family of generalized Bowen--Series boundary maps studied in~\cite{KU17,AK19,A20,AKU-Flexibility} and defined by the formula
\< \label{fA definition}
    f_\A (x)=T_k(x) \quad\text{if } x \in [A_k,A_{k+1}),
\>
where
\[ \A=\{A_1,A_2,\dots,A_{8g-4}\}\quad\text{and}\quad A_k\in [P_k,Q_k]. \]
When all $A_k = P_k$ we denote the map by $f_\P$ (this map is what Adler and Flatto~\cite{AF91} refer to as ``the Bowen--Series boundary map,'' although Bowen and Series' construction~\cite{BS79} used $4g$-sided polygons).
In~\cite{AKU-Flexibility} we analyzed how the measure-theoretic entropy with respect to the smooth invariant measure~$\mu_\A$ of maps in this family changes in the Teichm\"uller space of $S$ and proved a flexibility result: the entropy $h_{\mu_\A}(f_\A)$ takes all values between $0$ and a maximum that is achieved on the surface that admits a regular $(8g-4)$-sided fundamental polygon. In contrast, the main result of this paper is rigidity of topological entropy: its value depends only on the genus of the surface, remains constant in the Teichm\"uller space $\Tc(S)$, and does not depend on the (multi-)parameter $\bar A$.

\begin{samepage}
\begin{thm}[Main Theorem] \label{thm main} 
    Let $\Gamma$ be a cocompact torsion free Fuchsian group such that $S=\Gamma\backslash\D$ is a surface of genus $g\ge 2$. For any $\bar A = \{A_1,\dots,A_{8g-4}\}$ with $A_k \in [P_k,Q_k]$, the map $f_\A : \Sb \to \Sb$ has topological entropy $h_\mathrm{top}(f_\A) = \log( 4g-3 + \sqrt{(4g-3)^2-1} )$.\footnote{The quantity $\log( 4g-3 + \sqrt{(4g-3)^2-1} )$ can also be expressed as $\operatorname{arccosh}(4g-3)$, but log\-arithm expressions are more common for entropies in general and especially for shifts, so we use the longer expression.}
\end{thm}
\end{samepage}

\begin{remark*} 
Most previous results on boundary maps $f_\A:\Sb \to \Sb$ require the parameters~$\A$ to be in a smaller class: \cite{AF91} uses only $\A=\bar P$ and $\A=\bar Q$, \cite{A20} focuses on extremal parameters and their duals, and~\cite{AK19,AKU-Flexibility} require that the parameters have the short cycle property. In this paper \Cref{thm main} applies to \emph{all} parameters~$\A$ with~$A_k \in [P_k,Q_k]$. Although our result shows that all maps $f_\A$ have the same topological entropy for a given genus $g$, they are not necessarily topologically conjugate, since, according to \cite{KU17}, the combinatorial structure of the orbits associated  to the discontinuity points $A_k$ could differ.
\end{remark*}

The paper is organized as follows. In Sections~\ref*{sec generators}--\ref*{sec conjugacy} we restrict ourselves to the case when $\Gamma$ admits a regular $(8g-4)$-sided fundamental polygon. In \Cref{{sec generators}}
we give the formulas for generators $T_k$ as functions on $\overline\D \subset \C$ (\Cref{thm Tk formulas}) and prove two additional symmetric properties of generators as functions on $\partial \D$. In \Cref{sec Markov} we compute the maximal eigenvalue of the transition matrices for all ``extremal'' parameters and hence the topological entropy for these Markov cases. In \Cref{sec conjugacy} we prove some symmetric properties of the map $\psi_\P$ conjugating $f_\P$ to a constant slope map. We conclude that $\psi_\P$ actually conjugates all $f_\A$ to constant slope maps, and in \Cref{sec end} we use this to prove \Cref{thm main}, first for $\Gamma$ admitting regular $(8g-4)$-sided fundamental polygons and then in the fully general case. A technical result stated and used in \Cref{sec conjugacy} is proved in \Cref{sec appendix}.

\section{Additional properties of generators}\label{sec generators}

\begin{prop}\label{thm Tk formulas}
If the $(8g-4)$-sided fundamental polygon $\Fc$ is regular, then the generators $T_k$ of the group $\Gamma$ are given as functions on $\overline\D \subset \C$ by
\< \label{T} T_k(z) = (-1)^{k+1} \frac{e^{\i(1-k)\alpha}z + \i \sqrt{\cos\alpha}}{(-\i \sqrt{\cos\alpha})z + e^{\i(k-1)\alpha}}, \quad\text{where }\alpha := \frac{2\pi}{8g-4}. \>
\end{prop}

\begin{proof}
We derive a formula for $T_k(z)$ based on some geometric considerations also presented in~\cite[Section~4.3]{K92} and \cite[Appendix]{KU17}.

Let $T_k(z)=(az+\overline c)/(c z+\overline a)$, where $\abs{a}^2-\abs{c}^2=1$. The isometric circle $P_kQ_{k+1}$ of $T_k$, also denoted $I(T_k)$, is given by the equation $\abs{c z+\overline a}=1$, has center $O_k$ located at $-\overline a/c$ with $\arg(-\overline a/c)=-\frac{\pi}{2}+(k-1)\alpha$ and radius $R=1/\abs{c}$.

Let $d=\abs{a}/\abs{c}$ be the distance from the origin $O$ to the center $O_k$ of $I(T_k)$. The following formula for $R$ was obtained in \cite[Appendix]{KU17}: 
\[
R=\frac{\sqrt{2}\sin(\alpha/2)}{\sqrt{\cos\alpha}}=
\frac{\sqrt{1-\cos\alpha}}{\sqrt{\cos\alpha}}.
\]
This implies that \[ \abs c = \frac{1}{R} = \frac{\sqrt{\cos \alpha }}{\sqrt{1-\cos \alpha}} \quad \text{ and } \quad \abs a = d\abs{c}=\frac{1}{\sqrt{1-\cos \alpha}}. \]

The isometric circle $I(T_k)$ is mapped by $T_k$ to the isometric circle of $T_k^{-1}=T_{\sigma(k)}$ with center located at $a/c$. We analyze two cases:
\begin{itemize}
\item If $k$ is odd and $k<4g$, then $T_k^{-1}=T_{4g-k}$. The oriented angle $\angle O_kOO_{4g-k}=(4g-2k)\alpha=\pi+(2-2k)\alpha$, so $a/c=e^{\i(\pi+(2-2k)\alpha)}(-\overline a /c)$, which implies that $\arg(a)=\pi+(1-k)\alpha$. From $\arg(-\overline a/c)=-\frac{\pi}{2}+(k-1)\alpha$, we get $\arg(c) = \pi/2$.~Thus,
\[ a=\dfrac{-e^{\i(1-k)\alpha}}{\sqrt{1-\cos\alpha}} \qquad\text{and}\qquad c=\dfrac{\i\sqrt{\cos(\alpha)}}{\sqrt{1-\cos\alpha}}, \]
and, after simplifying the common term $-\sqrt{1-\cos \alpha}$, we get relation~\eqref{T}.
\item  If $k$ is odd and $k>4g$, then $\sigma(k)=4g-k\: (\text{mod } 8g-4)=12g-k-4$, so $T_k^{-1}=T_{12g-k-4}$. The oriented angle $\angle O_kOO_{12g-k-4}=(12g-2k-4)\alpha=3\pi+(2-2k)\alpha$, so $a/c=e^{\i(3\pi+(2-2k)\alpha)}(-\overline a /c)$, which implies that $\arg(a)=2\pi+(1-k)\alpha$. From $\arg(-\overline a/c)=-\frac{\pi}{2}+(k-1)\alpha$, we get $\arg(c) = -\pi/2$. Thus
\[ a=\dfrac{e^{\i(1-k)\alpha}}{\sqrt{1-\cos\alpha}} \qquad\text{and}\qquad c=-\dfrac{\i\sqrt{\cos(\alpha)}}{\sqrt{1-\cos\alpha}}, \]
and, after simplifying the common term $\sqrt{1-\cos \alpha}$, we get relation~\eqref{T}.
\end{itemize}
The case when $k$ is even can be treated similarly.
\end{proof}

\begin{prop}\label{thm Tk shifting} 
For all $x \in \Sb$, $T_k(x+\alpha)=T_{k-1}(x)+(4g-3)\alpha$.
\end{prop}

\begin{proof}
Let $\beta = (4g-3)\alpha$. Then in complex (multiplicative) notation, the claim is \[ T_k(e^{\i\alpha}z) = e^{\i\beta} T_{k-1}(z). \]
Note that 
$ 
	e^{\i\beta} 
	= e^{\i(\frac{4g-3}{8g-4})2\pi}
	= -e^{-\i\alpha}
$. 
Then, using~\eqref{T},
\begin{align*}
	T_k(e^{\i\alpha}z)
	&= (-1)^{k+1} \frac{(e^{\i(1-k)\alpha})(e^{\i\alpha}z) + \i \sqrt{\cos\alpha}}{(-\i\sqrt{\cos\alpha})(e^{\i\alpha}z) + e^{\i(k-1)\alpha}} \cdot \frac{-e^{\i\beta}}{e^{-\i\alpha}} \\
	&= (-1)^{k+1} (-e^{\i\beta}) \frac{(e^{\i(1-k)\alpha} e^{\i\alpha})z + \i \sqrt{\cos\alpha}}{(-\i\sqrt{\cos\alpha}) + (e^{\i(k-1)\alpha} e^{-\i\alpha})} \\
	&= (-1)^k (e^{\i\beta}) \frac{(e^{\i(-k)\alpha})z + \i \sqrt{\cos\alpha}}{(-\i\sqrt{\cos\alpha})z + e^{\i(k-2)\alpha}} \\
	&= e^{\i\beta} \cdot T_{k-1}(z). \qedhere
\end{align*}
\end{proof}

A function $r(x)$ is said to be \emph{centrally symmetric around~$c$} if $r(c + x) + r(c - x)$ is constant for all $x\in\Sb$ (this constant will be $2r(c)$).
This property is equivalent to saying that the graph of a lift of $r$ to $\R$ restricted to any rectangle $[c-\delta,c+\delta] \times [r(c-\delta),r(c+\delta)]$ is symmetric under rotation by $\pi$ around the center of that rectangle. If the circle is modeled as $\partial\D \subset \C$, then the analogous property is that $r(c\,z)\cdot r(c/z)$ is constant for all $z \in \partial\D$. \medskip

Denote by $C_k$ the midpoint of the segment $[P_k,Q_{k+1}]$. The next \namecref{thm Tk symmetry} asserts that the graph of $T_k$ is centrally symmetric around $C_k$.

\begin{prop} \label{thm Tk symmetry} For all $x \in \Sb$, $T_k(C_k + x) + T_k(C_k - x) = -2 C_k$.
\end{prop}
\begin{proof}
In complex (multiplicative) notation, the claim is that $T_k(C_k\,z) \cdot T_k(C_k/z) = 1/C_k^2$ for all $z \in \partial\D$. Since  $C_k \in \partial\D$ is the midpoint of the counter-clockwise arc of the circle from $P_k$ to $Q_{k+1}$, it satisfies $C_k^2 = P_k\cdot Q_{k+1}$ as complex numbers.
The isometric circle of $T_k(z)$ connects $P_k$ to $Q_{k+1}$ and consists of those $z \in \overline\D$ for which $\abs{T_k'(z)}=1$, so  $P_k, Q_{k+1} \in \partial\D$ are the complex numbers~$z$, $\abs z=1$, satisfying
\[ \abs{ \frac{-\i\sqrt{\cos\alpha}}{\sqrt{1-\cos\alpha}}\, z + \frac{e^{\i(k-1)\alpha}}{\sqrt{1-\cos\alpha}} } = 1. \]
The solutions to this equation are $-e^{\i(k-1)\alpha} w$ and $e^{\i(k-1)\alpha} \overline w$, where $w = \sqrt{1-\cos\alpha} + \i\sqrt{\cos\alpha}$. 
The product of the two solutions is
\begin{align*}
    P_k \cdot Q_{k+1}
    &= e^{\i(k-1)\alpha} w \cdot e^{\i(k-1)\alpha} (-\overline w) 
    = (e^{\i(k-1)\alpha})^2 (-\abs{w}^2) \\&
    = e^{-\i\pi} e^{\i(2k-2)\alpha}
    = e^{\i((2k-2)\alpha-(4g-2)\alpha)} 
	= (e^{\i\alpha})^{2k-4g}.
\end{align*}
Since $C_k$ is the midpoint of the smaller of the two arcs comprising $\partial\D \!\setminus\! \{P_k,Q_{k+1}\}$, we have that \< \label{Ck} C_k = (e^{\i\alpha})^{k-2g}. \>
To prove \Cref{thm Tk symmetry}, we use the alternative form
\[ C_k = e^{\i(k-2g)\alpha} = e^{-\i(2g-1)\alpha} e^{\i(k-1)\alpha} = e^{-\i(\pi/2)} e^{\i(k-1)\alpha} = -\i  e^{\i(k-1)\alpha} \]
to compute
\begin{align*}
    T_k(C_k\,z)
    &= (-1)^{k+1} \frac{e^{\i(1-k)\alpha}(-\i  e^{\i(k-1)\alpha}z) + \i \sqrt{\cos\alpha}}{(-\i \sqrt{\cos\alpha})(-\i  e^{\i(k-1)\alpha}z) + e^{\i(k-1)\alpha}} \\*
    &= \frac{(-1)^{k+1} \cdot \i}{ e^{\i(k-1)\alpha} } \cdot \frac{-z + \sqrt{\cos\alpha}}{(-\sqrt{\cos\alpha})z + 1} 
    = \frac{(-1)^k}{C_k} \cdot \frac{z - \sqrt{\cos\alpha}}{(-\sqrt{\cos\alpha})z + 1} 
\end{align*}
and then
\begin{align*}
    T_k(C_k z) \cdot T_k(C_k/z) 
    &= \!\left(\! \frac{(-1)^k}{C_k} \frac{z - \sqrt{\cos\alpha}}{(-\sqrt{\cos\alpha})z + 1} \right)\!\!\left(\! \frac{(-1)^k}{C_k} \frac{z^{-1} - \sqrt{\cos\alpha}}{(-\sqrt{\cos\alpha})z^{-1} + 1} \right)\! 
    = \frac1{C_k^2}
\end{align*}
as claimed.
\end{proof}

\begin{cor} \label{thm Tk oddness}
    For all $x \in \Sb$, $T_k(-x) = -T_{4g-k}(x)$.
\end{cor}

\begin{proof}
Applying \Cref{thm Tk shifting} repeatedly gives
\[ T_k(x) = T_{k+n}(x+n\alpha) - n\beta \]
for any $n \in \Z$, where $\beta = (4g-3)\alpha$. Using $n = 2g-k$ we have
\begin{align*}
    T_k(-x) &= T_{2g}\big({-}x + (2g-k)\alpha\big) - (2g-k)\beta \\*
    &= -T_{2g}\big(x - (2g-k)\alpha\big) - (2g-k)\beta \quad\text{by \Cref{thm Tk symmetry}} \\*
    &= -T_{2g-(k-2g)}(x) + (2g-k)\beta - (2g-k)\beta 
    = -T_{4g-k}(x).\qedhere
\end{align*}
\end{proof}

\section{Markov matrices for extremal parameters}\label{sec Markov}

\begin{defn} A parameter choice
$\bar A = \{A_1,\dots,A_{8g-4}\}$ with $A_k \in [P_k,Q_k]$ is called \emph{extremal} if for each $k$ either $A_k = P_k$ or $A_k = Q_k$.
\end{defn}
Extremal parameters were first introduced in~\cite{A20}, in which several results of~\cite{KU17,KU17e,AK19} for parameters with ``short cycles'' were extended to extremal parameters. Note that the classical cases $\A = \P$ and $\A = \Q$ are examples of extremal parameter choices.

Since for all $k=1,\dots,8g-4$, $f_\A(P_k)$ and $f_\A(Q_k)$ belong to the set $\bar P\cup\bar Q$ (see \cite[Proposition~2.2]{KU17}, originally \cite[Theorem~3.4]{AF91}), the partition of $\Sb$ into intervals $I_1,\dots,I_{16g-8}$ given by
\[ I_{2k-1} := [P_k, Q_k], \qquad I_{2k} := [Q_{k}, P_{k+1}], \qquad k=1,\dots,8g-4, \]
is a Markov partition for $f_\A$ for every extremal $\A$. Each extremal $\A$ has a transition matrix $M_\A = (m_{i,j})$ with
\[ m_{i,j} := \left\{\begin{array}{ll} 1 &\text{if } f_\A(I_i) \supset I_j \\ 0 &\text{otherwise,} \end{array}\right. \]
and an infinite sequence $(\omega_0,\omega_1,\dots)$ or finite sequence $(\omega_0,\omega_1,\dots,\omega_n)$ over the alphabet $\{1,\dots,16g-8\}$ is called \emph{$\A$-admissible} if all $m_{\omega_i,\omega_{i+1}}=1$ for all $i \ge 0$ (and $i < n$ for finite).

For each extremal parameter $\A$, we can define the shift space\footnote{In~\cite{AK19} the notation $X_\A$ is used for a space of sofic sequences in $8g-4$ symbols. Here we use it for a Markov shift on $16g-8$ symbols.} \[ X_\A  = \overline{ \big\{\, \omega \in \{1,...,16g-8\}^\N : \text{$\omega$ is admissible} \,\big\} } \]
on which we have the left-shift $\sigma_\A : X_\A \to X_\A$ and the essentially bijective \emph{coding map} $\phi_\A: X_\A \to \Sb$ given by \< \label{coding map} \phi_\A(\omega) = \bigcap_{i=0}^\infty f_\A^{-i}(I_{\omega_i}) \>
so that the following diagram is commutative:
\[ \begin{tikzcd}
  X_\A \arrow{r}{\sigma_\A} \arrow[swap]{d}{\phi_\A}
  & X_\A \arrow{d}{\phi_\A} \\
  \Sb \arrow{r}{f_\A}
  & || \Sb 
\end{tikzcd}. \]
In the case where $f_\A$ is Markov, the system $(X_\A, \sigma_\A)$ is a topological Markov chain.

\medskip
The following formulas use \cite[Proposition~3.1 and Lemma~3.2]{AK19}.
For odd indices $2k-1$, depending on whether $A_k = P_k$ or $A_k = Q_k$ we have, respectively, either
\< \label{Markov odd P}
	f_\A(I_{2k-1}) 
    = T_k(I_{2k-1}) 
    = [Q_{\sigma(k)+1},Q_{\sigma(k)+2}]
    = I_{2\sigma(k)+2} \cup I_{2\sigma(k)+3}
\>
or 
\< \label{Markov odd Q}
    \begin{split}
	f_\A(I_{2k-1}) 
	= T_{k-1}(I_{2k-1})
	&= [P_{\sigma(k-1)-1},P_{\sigma(k-1)}]
	= [P_{\sigma(k)+4g-2},P_{\sigma(k)+4g-1}] \\
    &= I_{2\sigma(k)+8g-5} \cup I_{2\sigma(k)+8g-4}.
\end{split}
\>
In either case, $f_\A(I_{2k-1})$ is the union of two consecutive Markov partition elements.
For even indices $2k$, we know that for any extremal $\A$ the image
\< \label{Markov even}
    \begin{split}
	f_\A(I_{2k}) 
	&= T_k(I_{2k})
	= [Q_{\sigma(k)+2}, P_{\sigma(k)-1}] \\ &
    = I_{2\sigma(k)+4} \cup I_{2\sigma(k)+5} \cup \cdots \cup I_{2\sigma(k)-4}
    = \Sb \setminus \bigcup_{\ell=2\sigma(k)-3}^{2\sigma(k)+3} I_\ell
\end{split}
\>
is the union of $(16g-8)-7 = 16g-15$ consecutive intervals on the circle. Recall that the indices are mod $16g-8$; for example, with $g=2$ and $k=1$ we get
\[ f_\A(I_{2}) = I_{18} \cup I_{19} \cup \cdots \cup I_{10} = I_{18} \cup I_{19} \cup \cdots \cup I_{24} \cup I_{1} \cup \cdots \cup I_{10}. \]
The matrices $M_\P$ and $M_\Q$ for genus 2 are shown in \Cref{fig matrices}.

\begin{figure}[htb]
\providecommand\0{{\color{black}0}}
\providecommand\1{{\color{black}\!\bf1\!}}

{\setlength\arraycolsep{1.8pt}
\tiny$\left(\begin{array}{cccccccccccccccccccccccc}
\0 & \0 & \0 & \0 & \0 & \0 & \0 & \0 & \0 & \0 & \0 & \0 & \0 & \0 & \0 & \1 & \1 & \0 & \0 & \0 & \0 & \0 & \0 & \0 \\
\1 & \1 & \1 & \1 & \1 & \1 & \1 & \1 & \1 & \1 & \0 & \0 & \0 & \0 & \0 & \0 & \0 & \1 & \1 & \1 & \1 & \1 & \1 & \1 \\
\0 & \1 & \1 & \0 & \0 & \0 & \0 & \0 & \0 & \0 & \0 & \0 & \0 & \0 & \0 & \0 & \0 & \0 & \0 & \0 & \0 & \0 & \0 & \0 \\
\0 & \0 & \0 & \1 & \1 & \1 & \1 & \1 & \1 & \1 & \1 & \1 & \1 & \1 & \1 & \1 & \1 & \1 & \1 & \1 & \0 & \0 & \0 & \0 \\
\0 & \0 & \0 & \0 & \0 & \0 & \0 & \0 & \0 & \0 & \0 & \1 & \1 & \0 & \0 & \0 & \0 & \0 & \0 & \0 & \0 & \0 & \0 & \0 \\
\1 & \1 & \1 & \1 & \1 & \1 & \0 & \0 & \0 & \0 & \0 & \0 & \0 & \1 & \1 & \1 & \1 & \1 & \1 & \1 & \1 & \1 & \1 & \1 \\
\0 & \0 & \0 & \0 & \0 & \0 & \0 & \0 & \0 & \0 & \0 & \0 & \0 & \0 & \0 & \0 & \0 & \0 & \0 & \0 & \0 & \1 & \1 & \0 \\
\1 & \1 & \1 & \1 & \1 & \1 & \1 & \1 & \1 & \1 & \1 & \1 & \1 & \1 & \1 & \1 & \0 & \0 & \0 & \0 & \0 & \0 & \0 & \1 \\
\0 & \0 & \0 & \0 & \0 & \0 & \0 & \1 & \1 & \0 & \0 & \0 & \0 & \0 & \0 & \0 & \0 & \0 & \0 & \0 & \0 & \0 & \0 & \0 \\
\1 & \1 & \0 & \0 & \0 & \0 & \0 & \0 & \0 & \1 & \1 & \1 & \1 & \1 & \1 & \1 & \1 & \1 & \1 & \1 & \1 & \1 & \1 & \1 \\
\0 & \0 & \0 & \0 & \0 & \0 & \0 & \0 & \0 & \0 & \0 & \0 & \0 & \0 & \0 & \0 & \0 & \1 & \1 & \0 & \0 & \0 & \0 & \0 \\
\1 & \1 & \1 & \1 & \1 & \1 & \1 & \1 & \1 & \1 & \1 & \1 & \0 & \0 & \0 & \0 & \0 & \0 & \0 & \1 & \1 & \1 & \1 & \1 \\
\0 & \0 & \0 & \1 & \1 & \0 & \0 & \0 & \0 & \0 & \0 & \0 & \0 & \0 & \0 & \0 & \0 & \0 & \0 & \0 & \0 & \0 & \0 & \0 \\
\0 & \0 & \0 & \0 & \0 & \1 & \1 & \1 & \1 & \1 & \1 & \1 & \1 & \1 & \1 & \1 & \1 & \1 & \1 & \1 & \1 & \1 & \0 & \0 \\
\0 & \0 & \0 & \0 & \0 & \0 & \0 & \0 & \0 & \0 & \0 & \0 & \0 & \1 & \1 & \0 & \0 & \0 & \0 & \0 & \0 & \0 & \0 & \0 \\
\1 & \1 & \1 & \1 & \1 & \1 & \1 & \1 & \0 & \0 & \0 & \0 & \0 & \0 & \0 & \1 & \1 & \1 & \1 & \1 & \1 & \1 & \1 & \1 \\
\1 & \0 & \0 & \0 & \0 & \0 & \0 & \0 & \0 & \0 & \0 & \0 & \0 & \0 & \0 & \0 & \0 & \0 & \0 & \0 & \0 & \0 & \0 & \1 \\
\0 & \1 & \1 & \1 & \1 & \1 & \1 & \1 & \1 & \1 & \1 & \1 & \1 & \1 & \1 & \1 & \1 & \1 & \0 & \0 & \0 & \0 & \0 & \0 \\
\0 & \0 & \0 & \0 & \0 & \0 & \0 & \0 & \0 & \1 & \1 & \0 & \0 & \0 & \0 & \0 & \0 & \0 & \0 & \0 & \0 & \0 & \0 & \0 \\
\1 & \1 & \1 & \1 & \0 & \0 & \0 & \0 & \0 & \0 & \0 & \1 & \1 & \1 & \1 & \1 & \1 & \1 & \1 & \1 & \1 & \1 & \1 & \1 \\
\0 & \0 & \0 & \0 & \0 & \0 & \0 & \0 & \0 & \0 & \0 & \0 & \0 & \0 & \0 & \0 & \0 & \0 & \0 & \1 & \1 & \0 & \0 & \0 \\
\1 & \1 & \1 & \1 & \1 & \1 & \1 & \1 & \1 & \1 & \1 & \1 & \1 & \1 & \0 & \0 & \0 & \0 & \0 & \0 & \0 & \1 & \1 & \1 \\
\0 & \0 & \0 & \0 & \0 & \1 & \1 & \0 & \0 & \0 & \0 & \0 & \0 & \0 & \0 & \0 & \0 & \0 & \0 & \0 & \0 & \0 & \0 & \0 \\
\0 & \0 & \0 & \0 & \0 & \0 & \0 & \1 & \1 & \1 & \1 & \1 & \1 & \1 & \1 & \1 & \1 & \1 & \1 & \1 & \1 & \1 & \1 & \1
\end{array}\right)\quad
\left(\begin{array}{cccccccccccccccccccccccc}
\1 & \1 & \0 & \0 & \0 & \0 & \0 & \0 & \0 & \0 & \0 & \0 & \0 & \0 & \0 & \0 & \0 & \0 & \0 & \0 & \0 & \0 & \0 & \0 \\
\1 & \1 & \1 & \1 & \1 & \1 & \1 & \1 & \1 & \1 & \0 & \0 & \0 & \0 & \0 & \0 & \0 & \1 & \1 & \1 & \1 & \1 & \1 & \1 \\
\0 & \0 & \0 & \0 & \0 & \0 & \0 & \0 & \0 & \0 & \1 & \1 & \0 & \0 & \0 & \0 & \0 & \0 & \0 & \0 & \0 & \0 & \0 & \0 \\
\0 & \0 & \0 & \1 & \1 & \1 & \1 & \1 & \1 & \1 & \1 & \1 & \1 & \1 & \1 & \1 & \1 & \1 & \1 & \1 & \0 & \0 & \0 & \0 \\
\0 & \0 & \0 & \0 & \0 & \0 & \0 & \0 & \0 & \0 & \0 & \0 & \0 & \0 & \0 & \0 & \0 & \0 & \0 & \0 & \1 & \1 & \0 & \0 \\
\1 & \1 & \1 & \1 & \1 & \1 & \0 & \0 & \0 & \0 & \0 & \0 & \0 & \1 & \1 & \1 & \1 & \1 & \1 & \1 & \1 & \1 & \1 & \1 \\
\0 & \0 & \0 & \0 & \0 & \0 & \1 & \1 & \0 & \0 & \0 & \0 & \0 & \0 & \0 & \0 & \0 & \0 & \0 & \0 & \0 & \0 & \0 & \0 \\
\1 & \1 & \1 & \1 & \1 & \1 & \1 & \1 & \1 & \1 & \1 & \1 & \1 & \1 & \1 & \1 & \0 & \0 & \0 & \0 & \0 & \0 & \0 & \1 \\
\0 & \0 & \0 & \0 & \0 & \0 & \0 & \0 & \0 & \0 & \0 & \0 & \0 & \0 & \0 & \0 & \1 & \1 & \0 & \0 & \0 & \0 & \0 & \0 \\
\1 & \1 & \0 & \0 & \0 & \0 & \0 & \0 & \0 & \1 & \1 & \1 & \1 & \1 & \1 & \1 & \1 & \1 & \1 & \1 & \1 & \1 & \1 & \1 \\
\0 & \0 & \1 & \1 & \0 & \0 & \0 & \0 & \0 & \0 & \0 & \0 & \0 & \0 & \0 & \0 & \0 & \0 & \0 & \0 & \0 & \0 & \0 & \0 \\
\1 & \1 & \1 & \1 & \1 & \1 & \1 & \1 & \1 & \1 & \1 & \1 & \0 & \0 & \0 & \0 & \0 & \0 & \0 & \1 & \1 & \1 & \1 & \1 \\
\0 & \0 & \0 & \0 & \0 & \0 & \0 & \0 & \0 & \0 & \0 & \0 & \1 & \1 & \0 & \0 & \0 & \0 & \0 & \0 & \0 & \0 & \0 & \0 \\
\0 & \0 & \0 & \0 & \0 & \1 & \1 & \1 & \1 & \1 & \1 & \1 & \1 & \1 & \1 & \1 & \1 & \1 & \1 & \1 & \1 & \1 & \0 & \0 \\
\0 & \0 & \0 & \0 & \0 & \0 & \0 & \0 & \0 & \0 & \0 & \0 & \0 & \0 & \0 & \0 & \0 & \0 & \0 & \0 & \0 & \0 & \1 & \1 \\
\1 & \1 & \1 & \1 & \1 & \1 & \1 & \1 & \0 & \0 & \0 & \0 & \0 & \0 & \0 & \1 & \1 & \1 & \1 & \1 & \1 & \1 & \1 & \1 \\
\0 & \0 & \0 & \0 & \0 & \0 & \0 & \0 & \1 & \1 & \0 & \0 & \0 & \0 & \0 & \0 & \0 & \0 & \0 & \0 & \0 & \0 & \0 & \0 \\
\0 & \1 & \1 & \1 & \1 & \1 & \1 & \1 & \1 & \1 & \1 & \1 & \1 & \1 & \1 & \1 & \1 & \1 & \0 & \0 & \0 & \0 & \0 & \0 \\
\0 & \0 & \0 & \0 & \0 & \0 & \0 & \0 & \0 & \0 & \0 & \0 & \0 & \0 & \0 & \0 & \0 & \0 & \1 & \1 & \0 & \0 & \0 & \0 \\
\1 & \1 & \1 & \1 & \0 & \0 & \0 & \0 & \0 & \0 & \0 & \1 & \1 & \1 & \1 & \1 & \1 & \1 & \1 & \1 & \1 & \1 & \1 & \1 \\
\0 & \0 & \0 & \0 & \1 & \1 & \0 & \0 & \0 & \0 & \0 & \0 & \0 & \0 & \0 & \0 & \0 & \0 & \0 & \0 & \0 & \0 & \0 & \0 \\
\1 & \1 & \1 & \1 & \1 & \1 & \1 & \1 & \1 & \1 & \1 & \1 & \1 & \1 & \0 & \0 & \0 & \0 & \0 & \0 & \0 & \1 & \1 & \1 \\
\0 & \0 & \0 & \0 & \0 & \0 & \0 & \0 & \0 & \0 & \0 & \0 & \0 & \0 & \1 & \1 & \0 & \0 & \0 & \0 & \0 & \0 & \0 & \0 \\
\0 & \0 & \0 & \0 & \0 & \0 & \0 & \1 & \1 & \1 & \1 & \1 & \1 & \1 & \1 & \1 & \1 & \1 & \1 & \1 & \1 & \1 & \1 & \1
\end{array}\right)
$}
\caption{Transition matrices $M_\P$ (left) and $M_\Q$ (right) for $g=2$}
\label{fig matrices}
\end{figure}

\begin{prop} \label{thm max eigenvalue} For any extremal $\A$, the maximal eigenvalue of~$M_\A$ is \[ \lambda = 4g-3 + \sqrt{(4g-3)^2-1}. \] \end{prop}

\begin{proof}
Gelfand's Formula~\cite{G41} states that $\lim_{n\to\infty} \norm{M_\A^n}^{1/n}$ equals the maximal eigenvalue of $M_\A$, where $\norm{\,\cdot\,}$ is any matrix norm.
The ``entrywise norm'' $\abs{M_A^n}$ given by the sum of (absolute values of) all entries in $M_\A^n$ counts the total number of admissible sequences of length $n+1$, that is,
\[
	\abs{M_\A^n} = \#\{\, (\omega_0,\omega_1,\dots,\omega_n) : m_{\omega_i,\omega_{i+1}} = 1 \text{ for } i=0,\dots,n-1 \,\}.
\]
To compress notation, we will write
\providecommand\Nall[1]{N_{#1}}
\providecommand\Neven[1]{N_{#1}^\mathrm{even}}
\providecommand\Nodd[1]{N_{#1}^\mathrm{odd}}
\[ \begin{split}
	\Nall n &:= \#\{\, (\omega_0,\omega_1,\dots,\omega_n) : \text{admissible} \,\} = \abs{M_\A^n} \\
	\Neven n &:= \#\{\, (\omega_0,\omega_1,\dots,\omega_n) :\text{admissible and } \omega_n\text{ is even} \,\} \\
	\Nodd n &:= \#\{\, (\omega_0,\omega_1,\dots,\omega_n) : \text{admissible and } \omega_n\text{ is odd} \,\}.
\end{split} \]
 Thus
\< \Nall{n+1} = 2 \Nodd{n} + (16g-15) \Neven{n}. \label{N n+1 decomp} \>

Since the indices of intervals that make up $f_\A(I_i)$ are consecutive, $f_\A(I_{2k-1})$ is the union of one even-index and one odd-index Markov interval, and $f_\A(I_{2k})$ is the union of $8g-8$ intervals with odd indices and $8g-7$ intervals with even indices. In terms of counting sequences,
\< \label{N more decomps} \begin{split}
	\Nodd{n} &= \Nodd{n-1} + (8g-8) \Neven{n-1} \\
	\Neven{n} &= \Nodd{n-1} + (8g-7) \Neven{n-1}.
\end{split} \>

Using~\eqref{N more decomps} and the fact that $\Nodd{n-1} + \Neven{n-1} = \Nall{n-1}$, we will convert~\eqref{N n+1 decomp} into a recurrence relation for $\Nall n$.
\begin{align*}
	\Nall{n+1} &= 2 \Nodd{n} + (16g-15) \Neven{n} \\*
	&= \big( (8g\!-\!6) - (8g\!-\!8) \big) \Nodd{n} + \big( (8g\!-\!6) + (8g\!-\!9) \big) \Neven{n} \\
	&= (8g\!-\!6) \Nall{n} - (8g\!-\!8) \Nodd{n} + (8g\!-\!9) \Neven{n} \\*
	&= (8g\!-\!6) \Nall{n} - (8g\!-\!8) \big( \Nodd{n-1} + (8g\!-\!8) \Neven{n-1} \big) \qquad\smash{\raisebox{-1em}{by~\eqref{N more decomps}}} \\*&\qquad + (8g\!-\!9) \big( \Nodd{n-1} + (8g\!-\!7) \Neven{n-1} \big) \\
	&= (8g\!-\!6) \Nall{n} + (-1) \Nodd{n-1} + (-1) \Neven{n-1} \\*
	&= (8g\!-\!6) \Nall{n} - \Nall{n-1}.
\end{align*}
Any nonzero sequence $(N_0,N_1,N_2,\dots)$ satisfying the linear recurrence relation
\[ \Nall{n+1} = K \Nall{n} - \Nall{n-1} \]
has an explicit expression of the form
\[ \Nall{n} = {c}_1 \cdot \Big(\!K/2 + \!\sqrt{(K/2)^2-1}\Big)^{\!n} + {c}_2 \cdot \Big(\!K/2 + \!\sqrt{(K/2)^2-1}\Big)^{\!-n} \]
for some constants ${c}_1$ and ${c}_2$, and therefore $\lim_{n\to\infty} (\Nall n)^{1/n} = K/2 + \!\sqrt{(K/2)^2-1}$.
For $\Nall n = \abs{M_\A^n}$ we have exactly this relation with $K=8g-6$; therefore the maximal eigenvalue of $M_\A$ is $\lim_{n\to\infty} \abs{M_\A^n}^{1/n} = 4g-3 + \sqrt{(4g-3)^2-1}$.
\end{proof}

\begin{cor}\label{thm entropy for extremal}
    For any extremal $\bar A$, $h_\text{top}(f_\A)=\log\lambda$.
\end{cor}

\begin{prop} \label{thm eigenvector} 
    For any extremal $\A$, the right eigenvector $v = (v_1,\dots,v_{16g-8})$, corresponding to the maximal eigenvalue $\lambda$, normalized so that $\sum {v_i} = 1$, is given by
	\[ \label{eigenvector} v_i = \Piecewise{ \vspace{0.8em} \dfrac{1}{\lambda(8g-4)} & \text{if $i$ is odd} \\ \dfrac{\lambda-1}{\lambda(8g-4)} &\text{if $i$ is even.} } \]
\end{prop}

\begin{proof} 
\renewcommand\xi{v'}
From the proof of \Cref{thm max eigenvalue}, for each odd $i$ the set $f_\A(I_i)$ is the union of two consecutive Markov partition elements, one with an even index and one with an odd index, and thus
\< \label{m fact odd} \sum_{k=1}^{8g-4} m_{i,2k-1} = 1 \quad\text{and}\quad \sum_{k=1}^{8g-4} m_{i,2k} = 1 \qquad\text{for odd $i$.} \>
Similarly, if $i$ is even then for any extremal $\A$ we know $f_\A(I_i)$ is the union of $8g-8$ odd indices and $8g-7$ even indices, so
\< \label{m fact even} \sum_{k=1}^{8g-4} m_{i,2k-1} = 8g-8 \quad\text{and}\quad \sum_{k=1}^{8g-4} m_{i,2k} = 8g-7 \qquad\text{for even $i$.} \>
We will show that the vector $\xi$ given by
\[ \label{eigenvector not normalized} \xi_i = \Piecewise{ \lambda-8g+7 & \text{if $i$ is odd} \\ 8g-8 &\text{if $i$ is even} } \]
satisfies $M_\A\xi = \lambda \xi$ by direct calculation.
First, note that $\lambda = 4g-3 + \sqrt{(4g\!-\!3)^2-1}$ is one root of the quadratic equation \< \label{lambda quadratic} \lambda(\lambda-8g+7) = \lambda-1. \>
Then we have
\begin{align*}
    (M_\A\xi)_i &= \sum_{j=1}^{16g-8} m_{i,j} \xi_j 
    = \sum_{k=1}^{8g-4} m_{i,2k-1} \xi_{2k-1} + \sum_{k=1}^{8g-4} m_{i,2k} \xi_{2k} \\
    &= (\lambda-8g+7)\left(\sum_{k=1}^{8g-4} m_{i,2k-1}\right) + (8g-8)\left(\sum_{k=1}^{8g-4} m_{i,2k}\right) \\*
    &= \Piecewise{
    	(\lambda\!-\!8g\!+\!7)(1) + (8g\!-\!8)(1) & \text{if $i$ is odd} \\
    	(\lambda\!-\!8g\!+\!7)(8g\!-\!8) + (8g\!-\!8)(8g\!-\!7) & \text{if $i$ is even}
	} \text{by~\eqref{m fact odd},\,\eqref{m fact even}} \\
	&= \Piecewise{
    	\lambda-1 & \text{if $i$ is odd} \\
    	\lambda(8g-8) & \text{if $i$ is even}
	} \\
	&= \Piecewise{ \lambda(\lambda-8g+7) & \text{if $i$ is odd} \\ \lambda(8g-8) &\text{if $i$ is even} } \text{~by~\eqref{lambda quadratic}} \\*
	&= \lambda \xi_i.
\end{align*}

The normalized eigenvector $v$ is then obtained by dividing $\xi$ by
\[
	\sum_{i=1}^{16g-8} {\xi_i}
	= (8g\!-\!4)(\lambda\!-\!8g\!+\!7) + (8g\!-\!4)(8g\!-\!8)
	= (8g-4)(\lambda-1).
\]
From~\eqref{lambda quadratic}, we have that
$\lambda(8g-4) = (\lambda+1)^2$
and so the coordinates of $v$ are
\[  
    v_i 
    = \frac{\lambda-8g+7}{(8g-4)(\lambda-1)}
    = \frac{(\lambda-1)/\lambda}{(8g-4)(\lambda-1)} 
    = \frac{1}{\lambda(8g-4)} 
\]
for odd $i$ and 
\[  
    v_i 
    = \frac{(8g-7)-1}{(8g-4)(\lambda-1)} 
    = \frac{(\lambda-\frac{\lambda-1}{\lambda})-1}{(8g-4)(\lambda-1)} 
    = \frac{1 - \frac1\lambda}{8g-4} 
    = \frac{\lambda-1}{\lambda(8g-4)}
\]
for even $i$.
\end{proof}

\section{Conjugacy to constant-slope map}\label{sec conjugacy}

We begin by stating a theorem combining several results of~\cite{Parry66,AM}, stated here for circle maps instead of interval maps (as in~\cite{MSz}):
\begin{thm}\label{newthm}
Given a piecewise monotone, piecewise continuous, topologically transitive map $f: \Sb \to \Sb$ of positive topological entropy $h>0$, there exists a unique (up to rotation of~$\Sb$) increasing homeomorphism $\psi:\Sb \to \Sb$ conjugating $f$ to a piecewise continuous map with constant slope $e^{h}$.
\end{thm}
Existence follows from \cite[Corollary B]{AM}, and uniqueness follows from \cite[Lemma 8.1, Theorem 8.2, Corollary 1]{AM}, where the continuity assumption is replaced by piecewise continuity (as in \cite{Hofbauer,Denker}).

The map $f_\P: \Sb \to \Sb$ is piecewise monotone, piecewise continuous, topologically transitive (see \cite[Lemma 2.5]{BS79}), and with positive topological entropy (see \Cref{thm entropy for extremal}), so by \Cref{newthm} there exists an increasing homeomorphism $\psi_\P : \Sb \to \Sb$ conjugating it to a map \[ \ell_\P := \psi_\P \circ f_\P \circ \psi_\P^{-1} \] with constant slope, see \Cref{fig fP and CS,fig psi}. The map $\psi_\P$ is unique up to rotation of~$\Sb$, and the slope of $\ell_\P$ is exactly $\lambda = e^{h_\mathrm{top}(f_\P)}$. Although the existence of a conjugacy to a constant-slope map holds for $f_\P$ associated to irregular fundamental polygons as well as regular, we will assume that $\Fc$ is regular for the remainder of this section.

\begin{figure}[ht]
    \includegraphics[width=\textwidth]{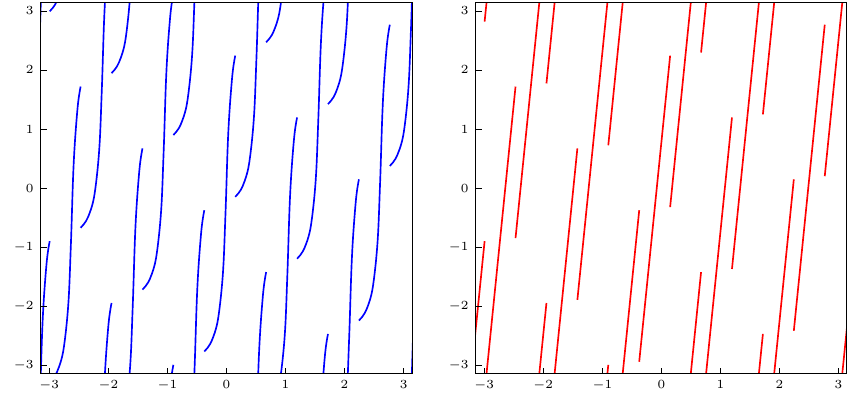}
    \caption{Plots of $f_\P(x)$ (left) and $\ell_\P(x)$ (right) for $g=2$}
    \label{fig fP and CS}
\end{figure}

\begin{figure}[ht]
    \includegraphics[width=\textwidth]{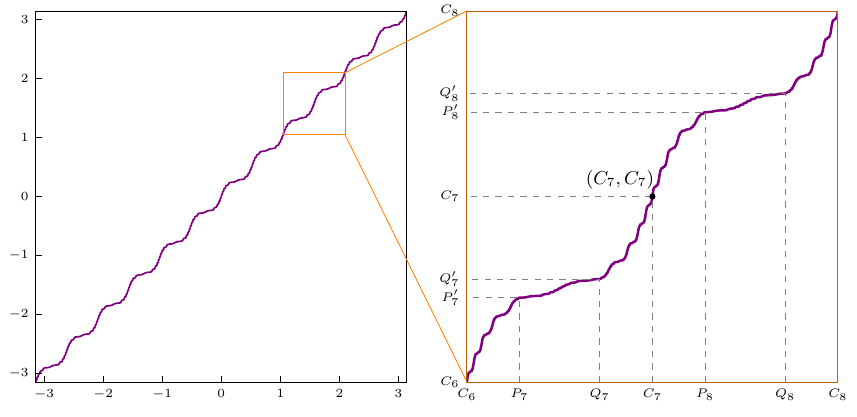}
    \caption{Plot of $\psi_\P(x)$ for $g=2$}
    \label{fig psi}
\end{figure}

The map $f_\Q: \Sb \to \Sb$, just like $f_\P,$ is piecewise monotone, piecewise continuous, topologically transitive, and with positive topological entropy, so by \Cref{newthm} there exists an increasing homeomorphism $\psi_\Q : \Sb \to \Sb$ conjugating it to a map $\ell_\Q$ of constant slope, unique up to rotation of $\Sb$. By \Cref{thm entropy for extremal}, both $\ell_\P$ and $\ell_\Q$ have the same slope.

Because $f_\P$ and $f_\Q$ are Markov maps, the conjugacies $\psi_\P$ and $\psi_\Q$ follow the classical construction due to Parry \cite{Parry64,Parry66} and used in the proof of \cite[Lem\-ma~5.1]{AM}.
For each extremal parameter $\A$, we define the probability measure $\rho_\A$ on $X_\A$ as follows: let $\lambda,v$ be the maximal eigenpair for the transition matrix $M_\A$; for an $\A$-admissible finite sequence $(\omega_0,...,\omega_n)$, we denote the symbolic cylinder 
\[ C_\A^{(\omega_0,\dots,\omega_n)} := \big\{\, \omega' \in X_\A : \omega'_i = \omega_i ~\forall~ 0 \le i \le n \,\big\} \]
and define the measure $\rho_\A$ of this cylinder as
\[
	\rho_\A\big(C_\A^{(\omega_0,\dots,\omega_n)}\big) = \frac{v_{\omega_n}}{\lambda^n}. 
\]
The measure $\rho_\A$ is equivalent to the shift-invariant ``Parry measure'' (the measure of maximal entropy; see \cite{Parry64,Parry66}).
The measure $\rho_\A$ is not shift-invariant but has the ``expanding property''
\[
\rho_\A(\sigma_\A(C))=\lambda\rho_A(C)
\]
for all cylinders $C$ on $(X_\A, \sigma_\A)$.

Using the measure $\rho_\A$, one constructs the push-forward measure $\rho'_\A$ on $\Sb$ given by
\[ \label{mu'} \rho'_\A(E) = \rho_\A\big( \phi^{-1}_\A(E) \big) \qquad\text{for a Borel set $E$}, \]
where $\phi_\A: X_\A \to \Sb$ is the symbolic coding map~\eqref{coding map}.
With the convention that $\psi_\A(0) = 0$, the conjugacy map $\psi_\A:\Sb \to \Sb$ is given by
\< \label{psi formula from Parry} \psi_\A(x) := 2\pi\cdot \Piecewise{
	\rho'_\A\big([0,x]\big) &\text{if }x \ge 0 \\
	-\rho'_\A\big([x,0]\big) &\text{if }x < 0
} \>
(the $2\pi$ appears because of our convention that the circle is $\Sb = [-\pi,\pi]$).

It turns out that the maps $\psi_\P$ and $\psi_\Q$ thus constructed coincide:
\begin{thm}\label{thm psi odd} For all $x\in\Sb$, $\psi_\P(x) = \psi_\Q(x)$. \end{thm}

To prove this, we need to connect the cylinder intervals of the two circle maps $f_\P$ and $f_\Q$. Given an $\A$-admissible sequence $\omega=(\omega_0,\omega_1,\dots,\omega_n)$ with $\omega_i \in \{1,\dots,16g-8\}$, we define the corresponding \emph{$\A$-cylinder interval}
\< \label{cylinder definition}
	\CylI\A{(\omega_0,\omega_1,\dots,\omega_n)}
	:= I_{\omega_0} \cap f_\A^{-1}(I_{\omega_1}) \cap \cdots \cap f_\A^{-n}(I_{\omega_n}).
\>

\begin{thm} \label{thm P as unions of Q}
	Let $\omega = (\omega_0,\dots,\omega_n)$ be $\P$-admissible. There exists a $\Q$-admissible sequence $(\eta_0,\dots,\eta_n, \eta_{n+1})$ such that $\eta_0 = \omega_0$, $\eta_{n+1}$ is odd, and
    \begin{enumerate}[\quad(i)] 
        \item \label{item recode 2} if $\omega_n$ is odd then 
        \< \label{recode 2} \CylI{\P}{\omega} = \CylI{\Q}{(\eta_0,\dots,\eta_n,\eta_{n+1})} \cup \CylI{\Q}{(\eta_0,\dots,\eta_n,\eta_{n+1}+1)}. \>
        
        \item \label{item recode 17} if $\omega_n$ is even then either 
        \< \label{recode 17} \CylI\P\omega = \CylI\Q{(\eta_0,\dots,\eta_n)} = \bigcup_{i=0}^{16g-14} \CylI\Q{(\eta_0,\dots,\eta_n,2\sigma(\eta_n/2)+4+i)} \>
        or
        \< \label{recode 2+15} \CylI\P\omega = \CylI\Q{(\eta_0,\dots,\eta_n,\eta_{n+1})} \cup \CylI\Q{(\eta_0,\dots,\eta_n,\eta_{n+1}+1)} \cup  \bigcup_{i=0}^{16g-16} \CylI\Q{(\eta_0,\dots,\eta_n+1,2\sigma(\frac{\eta_n+1}{2})+6+i)}. \>
    \end{enumerate}
\end{thm}
The proof of \Cref{thm P as unions of Q}, as well as the distinction between the two forms~\eqref{recode 17} and~\eqref{recode 2+15}, is rather technical and is left for \Cref{sec appendix}.

\begin{proof}[{Proof of \Cref{thm psi odd}}] \label{proof of psi odd}
Recall from \Cref{thm eigenvector} that for both $\A=\P$ and $\A=\Q$ the right-eigenvector $v$ of $M_\A$ corresponding to eigenvalue $\lambda$ is
\[ \label{v again} 
    v = c \cdot (1,\; \lambda\!-\!1,\; 1,\; \lambda\!-\!1,\; \dots \;,\; 1,\; \lambda\!-\!1), \]
where $c = 1/(\lambda(8g-4))$ corresponds to $\rho'_\A(\Sb) = 1$. We prove $\psi_\P = \psi_\Q$ by showing that $\rho_\P'(\CylI{\P}\omega) = \rho_\Q'(\CylI{\P}\omega)$ for all finite $\P$-admissible sequences~$\omega$.
Note that, because $\phi_\A^{-1}$ maps a cylinder interval to a symbolic cylinder, we have
\[ \label{cylinder measure}
    \rho'_\P(\CylI\P{(\omega_0,\ldots,\omega_n)}) = \frac{v_{\omega_n}}{\lambda^n} 
    \qquad\text{and}\qquad 
    \rho'_\Q(\CylI\Q{(\eta_0,\ldots,\eta_n)}) = \frac{v_{\eta_n}}{\lambda^n},
\]
where $(\omega_0,...,\omega_n)$ is $\P$-admissible and $(\eta_0,...,\eta_n)$ is $\Q$-admissible.

\medskip
Let $\omega = (\omega_0,\dots,\omega_n)$ be $\P$-admissible, and suppose $\omega_n$ is odd. Then $v_{\omega_n} = c$, and so
\[
	\rho_\P'(\CylI\P\omega)
	= \frac{v_{\omega_n}}{\lambda^n}
	= \frac{c}{\lambda^n}.
\]
By~\Cref{thm P as unions of Q},  
\[ \CylI\P\omega = \CylI\Q\eta \cup \CylI\Q{\eta'} \]
for some $\eta = (\eta_0,\dots,\eta_{n+1})$ and $\eta' = (\eta_0,\dots,\eta_n,\eta_{n+1}+1)$. Since $\eta_{n+1}$ and $\eta'_{n+1}$ have different parities, we know
\[ v_{\eta_{n+1}} + v_{\eta'_{n+1}} = c + (\lambda-1)c = \lambda c, \] 
and can compute
\begin{align*}
	\rho_\Q'(\CylI\P\omega)
	&= \rho_\Q'(\CylI\Q{\eta} \cup \CylI\Q{\eta'}) 
	= \frac{v_{\eta_{n+1}}}{\lambda^{n+1}} + \frac{v_{\eta'_{n+1}}}{\lambda^{n+1}}
	= \frac{1}{\lambda^{n+1}} (v_{\eta_{n+1}} + v_{\eta'_{n+1}})
	\\* &= \frac{1}{\lambda^{n+1}} (\lambda c) 
	= \frac{c}{\lambda^{n}} 
	= \rho_\P'(\CylI\P\omega).
\end{align*}

If instead $\omega_n$ is even, then \Cref{thm P as unions of Q} gives $\CylI\P\omega = \CylI\Q{\eta^{(1)}} \cup \cdots \cup \CylI\Q{\eta^{(16g-15)}}$ with exactly $8g-7$ of the final symbols $\eta^{(i)}_{n+1}$ being even (so $8g-8$ are odd). Therefore
\begin{align*}
	\rho_\Q'(\CylI\P\omega)
	&= \rho_\Q'(\CylI\Q{\eta^{(1)}} \cup \cdots \cup \CylI\Q{\eta^{(16g-15)}})
	= \frac{1}{\lambda^{n+1}} (v_{\eta^{(1)}_{n+1}} \!+ \cdots + v_{\eta^{(16g-15)}_{n+1}}) 
	\\* &= \frac{1}{\lambda^{n+1}} \big( (8g-7)(\lambda-1)c + (8g-8)c \big) 
	\\ &= \frac{c}{\lambda^{n+1}} \lambda(\lambda-1) \qquad\text{by~\eqref{lambda quadratic}} 
	\\* &= \frac{(\lambda-1)c}{\lambda^n}
	= \frac{v_{\omega_n}}{\lambda^n}
	= \rho_\P'(\CylI\P\omega),
\end{align*}
where $v_{\omega_n} = (\lambda-1)c$ because $\omega_n$ is even.

In both cases we have $\rho_\Q'(\CylI\P\omega) = \rho_\P'(\CylI\P\omega)$, and since $\{\, \CylI\P\omega : \omega \text{ is $\P$-admissible} \,\}$ generates all Borel sets in $\Sb$, the two measures $\rho_\P'$ and $\rho_\Q'$ on $\Sb$ are identical. From~\eqref{psi formula from Parry}, this implies that $\psi_\P = \psi_\Q$.
\end{proof}

For the remainder of~\Cref{sec conjugacy}, we deal almost exclusively with $\psi_\P$, although we will briefly invoke \Cref{thm psi odd}. We now show that $\psi_\P$ has translational (\Cref{thm psi periodic}) and central (\Cref{thm psi symmetry}) symmetry. Both of these properties can be seen in \Cref{fig psi}.

\begin{prop} \label{thm psi periodic} For all $x\in\Sb$, $\psi_\P(x + \alpha) = \psi_\P(x) + \alpha$. \end{prop}
\begin{proof}
Define $\phi:\Sb\to\Sb$ recursively by 
\[ \phi(x) := \Piecewise{ \psi_\P(x) &\text{if } x \in [P_1,P_2) \\ \phi(x-\alpha) + \alpha &\text{otherwise.} } \]
Thus $\phi(x+\alpha)=\phi(x)+\alpha$ for all $x$ by design, and since $\phi$ is increasing and continuous, we also have $\phi^{-1}(x+\alpha) = \phi^{-1}(x) + \alpha$ for all $x$.

Denote $x' = \phi(x)$.
By induction on $1 \le k \le 8g-4$, we will prove that there exists~$b_k$ such that $\phi \circ T_k \circ \phi^{-1}(x') = \lambda x'+b_k$ for $x' \in [P'_k,P'_{k+1}]$.
By construction the claim is true for $k=1$ since $\phi\big|_{[P_1,P_2]} = {\psi_\P}\big|_{[P_1,P_2]}$. 
Now assume it is true for $k$. Then by \Cref{thm Tk shifting}, writing $\beta = (4g-3)\alpha$, we have
\begin{align*}
	\phi \circ T_{k+1} \circ \phi^{-1}(x')
	&= \phi(T_{k+1}(x))
	= \phi(T_k(x-\alpha)+\beta)
	= \phi(T_k(x-\alpha)) + \beta
	\\&= \phi(T_k(\phi^{-1}(x'-\alpha))) + \beta
	= \lambda(x'-\alpha) + b_k + \beta
	\\&= \lambda x' + (b_k + \beta - \lambda \alpha)
\end{align*}
and so the claim holds for $k+1$ with $b_{k+1} = b_k + \beta - \lambda\alpha$.

Thus our map $\phi$, which satisfies $\phi(x+\alpha)=\phi(x)+\alpha$ for all $x$ by construction, conjugates $f_\P$ to a constant-slope map on $\Sb$. By the uniqueness of $\psi_\P$ (\Cref{newthm}), $\phi = \psi_\P$.
\end{proof}

Notice that $\psi_\P(x+n\alpha) = \psi_\P(x) + n\alpha$ for any integer $n$, as well as $\psi_\P^{-1}(x+n\alpha) = \psi_\P^{-1}(x)+n\alpha$.
Since $2\pi$ is an integer multiple of $\alpha$, $\psi_\P$ is well defined on $\Sb$, and we can choose the point where it is equal to $0$ at our convenience. 
We will assume that $\psi_\P$ fixes the point $C_{2g}=0$. Then \Cref{thm psi periodic} implies that $\psi_\P(C_k)=C_k$ for all~$k$.

\begin{prop} \label{thm psi symmetry} For all $x\in\Sb$, $\psi_\P(C_k+x)+\psi_\P(C_k-x) = 2C_k$. \end{prop}

\begin{proof}
	First, we prove that 
    \< \label{psi P to Q oddness} \psi_\Q(x) = -\psi_\P(-x). \>
    By \Cref{thm Tk oddness} we have $T_{k}(-x)=-T_{4g-k}(x)$. Since $f_\P$ acts by the generator $T_k$ on $x \in [P_k,P_{k+1}]$ and $f_\Q$ acts by the generator $T_{4g-k}$ on the reflected interval $-[P_k,P_{k+1}] = [Q_{4g-k},Q_{4g-k+1}]$, we have as a result that
    \[ f_{\bar Q}(x)=-f_{\bar P}(-x). \]
    
    To prove~\eqref{psi P to Q oddness}, set $\tilde x=\xi(x):=-\psi_\P(-x)$. Then  $-\tilde x=\psi_\P(-x)$ and $-x=\psi^{-1}_\P(-\tilde x)$, and then
    \begin{align*}
        \xi\circ f_\Q\circ\xi^{-1}(\tilde x)&=\xi(f_\Q(x))
        =\xi(-f_\P(-x))=-\psi_\P(f_\P(-x)) \\*
        &= -\psi_\P\circ f_\P\circ \psi^{-1}_\P(-\tilde x)
        = -\ell_\P(-\tilde x).
    \end{align*}
    Since this is a function with constant slope $\lambda$, the claim~\eqref{psi P to Q oddness} follows by uniqueness of the conjugacy.
    
    Combing~\eqref{psi P to Q oddness} with \Cref{thm psi odd}, we obtain $\psi_\P(-x) = -\psi_\P(x)$. Because $C_k = (k-2g)\alpha$, \Cref{thm psi periodic} implies 
    \[ \psi_\P(C_k+x) = C_k + \psi_\P(x) \]
    for all $k$. Therefore we compute that
    \[
        \psi_\P(C_k+x)+\psi_\P(C_k-x)
        = C_k + \psi_\P(x) + C_k + \psi_\P(-x)
        = 2C_k. \qedhere 
	\]
\end{proof}

A crucial observation for the proof of \Cref{thm main} is that $\psi_\P$ conjugates each $T_k$ as a function on the circle $\Sb$, and the resulting function \[ S_k := \psi_\P \circ T_k \circ \psi_\P^{-1} \]
consists of two linear pieces, one with slope $\lambda$ and the other with slope $\lambda^{-1}$.
See \Cref{fig T and S}, where 
\[ P'_k := \psi_\P(P_k) \quad\text{and}\quad Q'_k := \psi_\P(Q_k). \]

\begin{figure}[hb]
    \includegraphics[width=\textwidth]{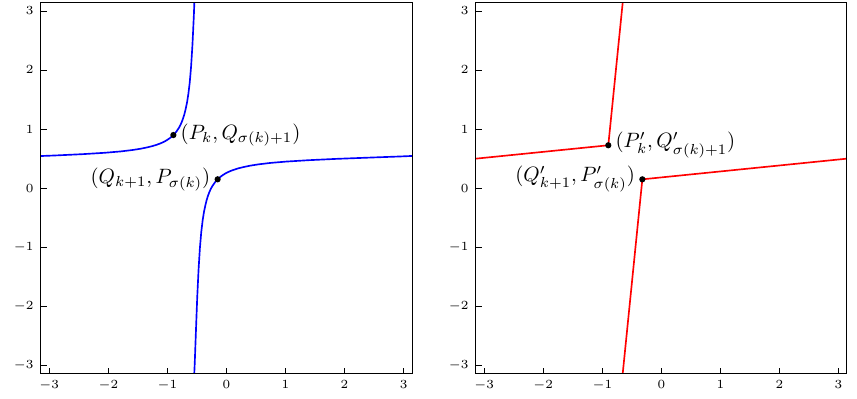}
    \caption{Plots of $T_k(x)$ (left) and $S_k(x)$ (right) with $k=3$ and $g=2$}
    \label{fig T and S}
\end{figure}


\begin{lem} \label{thm Sk linearity} The function $S_k : \Sb \to \Sb$ can be fully described as follows:
	\begin{enumerate}[\quad(a)]
		\item \label{Si-part-1} $S_k$ is linear on $[P'_k,Q'_{k+1}]$ with  slope $\lambda$;
		\item \label{Si-part-2} $S_k$ is linear on $[Q'_{k+1},P'_k]$ with slope $\lambda^{-1}$.
	\end{enumerate}
\end{lem}

\begin{proof}
(\ref{Si-part-1}) By construction, $S_k$ is linear on $[P'_k,P'_{k+1})$ since $[P_k,P_{k+1})$ is the interval where $f_\P$ acts as~$T_k$. 
Given the central symmetry of $T_k$ (\Cref{thm Tk symmetry}) and $\psi_\P$ (\Cref{thm psi symmetry}) around $C_k$, the composition $S_k = \psi_\P \circ T_k \circ \psi_\P^{-1}$ must also be symmetric around $C_k$. The image of $[P'_k,P'_{k+1}]$ under the symmetry $C_k+x \mapsto C_k-x$ is $[Q'_k,Q'_{k+1}]$, and thus $S_k$ is linear on $[Q'_k,Q'_{k+1}]$ with the same slope. Since the intervals of linearity $[P'_k,Q'_k]$ and $[Q'_k,Q'_{k+1}]$ overlap, there is no jump within their union, which is $[P'_k,Q'_{k+1}]$. We can in fact calculate 
\begin{align*}
    S_k(P'_k) &= \psi_\P(T_k(P_k)) = \psi_\P(Q_{\sigma(k)+1}) = Q'_{\sigma(k)+1} \\
    S_k(Q'_{k+1}) &= \psi_\P(T_k(Q_{k+1})) = \psi_\P(P_{\sigma(k)}) = P'_{\sigma(k)}
\end{align*} 
directly using \cite[Proposition~2.2]{KU17}.

\smallskip
(\ref{Si-part-2}) Because part~(\ref{Si-part-1}) holds for all $k$, we know $S_{\sigma(k)}$ maps $[P'_{\sigma(k)},Q'_{\sigma(k)+1}]$ linearly to $[Q'_{\sigma(\sigma(k))+1},P'_{\sigma(\sigma(k))}] = [Q'_{k+1},P'_k]$ with slope $\lambda$, and therefore $S_{\sigma(k)}^{-1}$ maps $[Q'_{k+1},P'_k]$ linearly to $[P'_{\sigma(k)},Q'_{\sigma(k)+1}]$ with slope $1/\lambda$. But
\[ S_{\sigma(k)}^{-1} = (\psi_\P \circ T_{\sigma(k)} \circ \psi_\P^{-1})^{-1} = \psi_\P \circ T_{\sigma(k)}^{-1} \circ \psi_\P^{-1} = \psi_\P \circ T_k \circ \psi_\P^{-1} \]
is exactly $S_k$.
\end{proof}

\section{Proof of Theorem~\ref{thm main}}\label{sec end}

We can now prove the rigidity of topological entropy, that is, $h_\mathrm{top}(f_\A)$ is the same for all pa\-ra\-me\-ters~$\bar A$ and for all fundamental polygons~$\Fc$.

\subsection*{Regular polygon} First we prove \Cref{thm main} in the case where $f_\A$ is associated to a regular $(8g-4)$-gon.
Let $\A = \{A_1,\dots,A_{8g-4}\}$ consist of any points satisfying $A_k \in [P_k,Q_k]$. 
Because $S_k = \psi_\P \circ T_k \circ \psi_\P^{-1}$ is linear on all of $[P'_k,Q'_{k+1}]$ with slope $\lambda$ by \Cref{thm Sk linearity}(\ref{Si-part-1}), the function $\psi_\P \circ f_\A \circ \psi_\P^{-1}$ (note the use of $f_\A$ with $\psi_\P$) is piecewise affine with constant slope~$\lambda$, and so, by \cite[Theorem~{$3'$}]{MSz} applied to such maps, the topological entropy of $f_\A$ is $\log\lambda$.

\subsection*{Teichm\"uller space}
As explained in \cite[Introduction]{AKU-Flexibility}, the Teichm\"uller space of a compact surface of genus $g$ may be viewed as the space of marked $(8g-4)$-fundamental polygons, and the partitions of the boundary $\Sb$ for various polygons are related via a homeomorphism of $\Sb$ by Fenchel--Nielsen Theorem.

Let $\tilde\Gamma$ be a Fuchsian group such that $\tilde\Gamma\backslash\D$ is a compact surface of genus $g$ whose fundamental $(8g-4$)-gon $\tilde\Fc$ is not regular. As explained in~\cite{KU17e}, there is a Fuchsian group $\Gamma$ having a regular fundamental $(8g-4)$-gon $\Fc$ and an orientation-preserving homeomorphism $h:\overline\D \to \overline\D$ such that $\tilde\Gamma=h\circ\Gamma\circ h^{-1}$.
Side $k$ of $\tilde\Fc$ extends to a geodesic $\tilde P_k \tilde Q_{k+1}$ and is glued to side $\sigma(k)$ by the map $\tilde T_k=h\circ T_k\circ h^{-1}$, where $\{T_k\}$ are generators of $\Gamma$ identifying the sides of $\Fc$.

For any $\tilde A=\{ \tilde A_1, ..., \tilde A_{8g-4} \}$ with $\tilde A_k  \in [\tilde P_k, \tilde Q_k]$, we define 
\newcommand\tfa{{\tilde f_{\tilde A}}}
\[ \tfa(x) := \tilde T_k(x) \quad\text{if }x \in [\tilde A_k,\tilde A_{k+1}). \]
Then the map $f_\A$ with $\A = \{ h^{-1}(\tilde A_1), \dots, h^{-1}(\tilde A_{8g-4}) \}$ is associated to the regular fundamental polygon, and (correcting a typo in~\cite{KU17e})
\[ \label{conj}
   \tfa =h\circ f_\A\circ h^{-1}.
\]
Since $\tfa$ is conjugate to $f_\A$, we conclude that $h_\mathrm{top}(\tfa) = h_\mathrm{top}(f_\A) = \log \lambda$, and this completes the proof of the main theorem. (In fact, the map $\psi \circ h^{-1}$ with $\psi = \psi_\P$ from \Cref{sec conjugacy} will conjugate $\tfa$ to a map of constant slope $\lambda$.)

\appendix
\section{Proof of Theorem~\ref{thm P as unions of Q}}\label{sec appendix}

We now prove~\Cref{thm P as unions of Q}, that is, that each cylinder interval $\CylI\P\omega$ can be written as unions of cylinder intervals $\CylI\Q\eta$ (see \Cref{fig cylinders} for the decompositions of $\CylI\P{(1,16)}$ and $\CylI\P{(1,17)}$). For the remainder of the \namecref{sec appendix}, we use the term ``cylinder'' (specifically, ``$\P$-cylinder'' and ``$\Q$-cylinder'') instead of ``cylinder interval'' for brevity.
\begin{figure}[hbt]
    \includegraphics[width=0.985\textwidth]{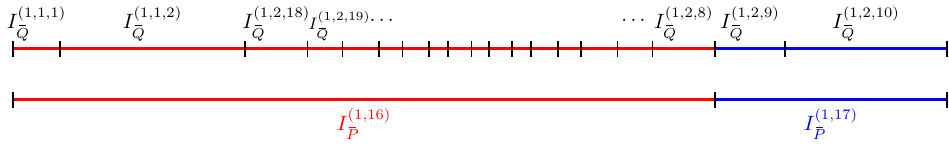}
    \caption{For $g=2$, $\CylI\P{(1,16)}$ as a union of $17$ $\Q$-cylinders (left) and $\CylI\P{(1,17)}$ as a union of two $\Q$-cylinders (right)}
    \label{fig cylinders}
\end{figure}

From~\eqref{cylinder definition}, we derive a common recursive description of cylinders:
	\[ 
		\CylI{\A}{(\omega_0,\omega_1,\dots,\omega_n)}
		= I_{\omega_0} \cap f_\A^{-1}(\CylI{\A}{(\omega_1,\dots,\omega_n)}).
	\]
	For our particular boundary maps, we have an alternative recursive relation: using the fact that each $T_k$ is bijective on all of $\Sb$, we can compute 
	\[ 
		\CylI{\P}{(\omega_0,\omega_1,\dots,\omega_n)}
		= T^{-1}_{\ceil{\omega_0/2}} (\CylI\P{(\omega_1,\dots,\omega_n)})
	\]
	for $\P$-admissible $\omega$ and
	\[ 
		\CylI{\Q}{(\omega_0,\omega_1,\dots,\omega_n)} 
		= T^{-1}_{\floor{\omega_0/2}} (\CylI\Q{(\omega_1,\dots,\omega_n)})
	\]
	for $\Q$-admissible $\omega$ without the need for an intersection.
	This is because $T^{-1}_{\ceil{\omega_0/2}}$ is contracting on $\CylI\P{(\omega_1,\dots,\omega_n)}$ and therefore $T^{-1}_{\ceil{\omega_0/2}}(\CylI\P{(\omega_1,\dots,\omega_n)})$ is already contained in $I_{\omega_0}$. 
	Note the use of ceiling~$\ceil{~~}$ for $\P$ and floor~$\floor{~~}$ for $\Q$, owing to the fact that $f_\P$ acts by $T_k$ on $I_{2k-1}$ and $I_{2k}$ while $f_\Q$ acts by $T_k$ on $I_{2k}$ and $I_{2k+1}$, and $k = \ceil{\frac{2k-1}2} = \ceil{\frac{2k}2} = \floor{\frac{2k}2} = \floor{\frac{2k+1}2}$.
	The formulas above can be extended recursively to
\begin{align*}
	\CylI{\P}{\omega} &= T^{-1}_{\ceil{\omega_0/2}} \circ T^{-1}_{\ceil{\omega_1/2}} \circ \cdots \circ T^{-1}_{\ceil{\omega_{n-1}/2}} (I_{\omega_n}) \quad\text{if $\omega$ is $\P$-admissible} \\
	\CylI{\Q}{\omega} &= T^{-1}_{\floor{\omega_0/2}} \circ T^{-1}_{\floor{\omega_1/2}} \circ \cdots \circ T^{-1}_{\floor{\omega_{n-1}/2}} (I_{\omega_n}) \quad\text{if $\omega$ is $\Q$-admissible}.
\end{align*}

\begin{samepage}
Equations~\eqref{Markov odd P},~\eqref{Markov odd Q}, and~\eqref{Markov even} can be interpreted as statements about admissible pairs of symbols:
\begin{itemize}
    \item In a $\P$-admissible sequence, an odd symbol $2k-1$ can only be followed by $2\sigma(k)+2$ or $2\sigma(k)+3$.
    \item In a $\Q$-admissible sequence, an odd symbol $2k-1$ can only be followed by $2\sigma(k)+8g-5$ or $2\sigma(k)+8g-4$.
    \item In a $\P$- or $\Q$-admissible sequence, an even symbol $2k$ can only be followed by a symbol from $\{2\sigma(k)+4, 2\sigma(k)+5, \dots, 2\sigma(k)-4\}$.
\end{itemize}
\end{samepage}

\smallskip
In the final item above, and in \Cref{thm admissibility helper} below, recall that these values are mod $16g-8$; see the explanation after \eqref{Markov even} on page~\pageref{Markov even}. The next lemma expands on the admissible pairs above and lists some longer admissible words used explicitly in the proof of \Cref{thm P as unions of Q}.

\smallskip
\begin{lem} ~ \label{thm admissibility helper}  \begin{enumerate}[\quad(a)]
    \item \label{item adm 3terms} For all $k$, if $\ell \in \{2k+8g, 2k+8g+1, \dots, 2k+8g-8\}$ then $(2k-1, 2\sigma(k)+8g-4, \ell)$ is $\Q$-admissible.
    \item \label{item adm 4terms} For all $k$, if $\ell \in \{ 2\sigma(k)+4$, $2\sigma(k)+5$, \dots, $2\sigma(k)-4 \}$ then $(2k-1, 2\sigma(k)+8g-5, 2k, \ell)$ is $\Q$-admissible.
    \item \label{item adm 5terms} For all $k$, if $\ell \in \{ 2k+8g-9, 2k+8g-8 \}$ then $(2k-1, 2\sigma(k)+8g-5, 2k-1, 2\sigma(k)+8g-4, \ell)$ is $\Q$-admissible.
\end{enumerate}
\end{lem}
The proof of \Cref{thm admissibility helper} consists of careful analysis of the transition matrix $M_\Q$ along with the useful identities
\[
	\sigma(k-1) = \sigma(k)-4g+3 \qquad\text{and}\qquad
	\sigma(k-2) = \sigma(k)+2,
\]
which follow by direct verification (see also \cite[Lemma~3.2]{AK19}).
\medskip

The following two lemmas establish some relations among the generators $\{T_k\}$ which will be used in the proof of \Cref{thm P as unions of Q}. We omit the composition notation (writing, e.g., $T^{-1}_k T^{-1}_{\sigma(k)+1}$, instead of $T^{-1}_k \circ T^{-1}_{\sigma(k)+1}$).

\begin{lem}[{\cite[Lemma~3.2]{KU17}}] \label{thm generator relation} $T^{-1}_k T^{-1}_{\sigma(k)+1} = T^{-1}_{k-1} T^{-1}_{\sigma(k)+4g-2}$. \end{lem}

\begin{lem} \label{thm generator inductive} For $m \ge 1$,
$
	T^{-1}_k \big( T^{-1}_{\sigma(k)+1} T^{-1}_{k+4g-1} \big)^m
	= \big( T^{-1}_{k-1} T^{-1}_{\sigma(k)+4g-3} \big)^m T^{-1}_k.
$
\end{lem}

\begin{proof}
The base case, $m=1$, is proven using \Cref{thm generator relation} twice, the second time for index $\sigma(k)+4g-2$:
\[
	T^{-1}_k T^{-1}_{\sigma(k)+1} T^{-1}_{k+4g-1}
	= T^{-1}_{k-1} T^{-1}_{\sigma(k)+4g-2} T^{-1}_{k+4g-1}
	= T^{-1}_{k-1} T^{-1}_{\sigma(k)+4g-3} T^{-1}_k.
\]
Then $m > 1$ follows by induction:
\begin{align*}
    T^{-1}_k \big( T^{-1}_{\sigma(k)+1} T^{-1}_{k+4g-1} \big)^m
    &= T^{-1}_k \big( T^{-1}_{\sigma(k)+1} T^{-1}_{k+4g-1} \big)^{m-1} T^{-1}_{\sigma(k)+1} T^{-1}_{k+4g-1} \\*
    &= \big( T^{-1}_{k-1} T^{-1}_{\sigma(k)+4g-3} \big)^{m-1} T^{-1}_k T^{-1}_{\sigma(k)+1} T^{-1}_{k+4g-1} \\
    &= \big( T^{-1}_{k-1} T^{-1}_{\sigma(k)+4g-3} \big)^{m-1} T^{-1}_{k-1} T^{-1}_{\sigma(k)+4g-3} T^{-1}_k \\*
    &= \big( T^{-1}_{k-1} T^{-1}_{\sigma(k)+4g-3} \big)^m T^{-1}_k.
    \qedhere
\end{align*}
\end{proof}

\medskip
We are now ready to proceed with an inductive proof of \Cref{thm P as unions of Q}, with the following refinement of part~(\ref{item recode 17}):

\phantomsection
\label{a-and-b}
\begin{enumerate}[\quad(a)]
    \item If all $\omega_k$ are even, or if $(\omega_{n-1},\omega_n)$ are even but not of the form $(2m,2\sigma(m)+4)$ for any $m$, then $\CylI\P\omega = \CylI\Q{(\eta_0,\dots,\eta_n)}$ with $\eta_n$ even, and therefore \[ \qquad\quad \CylI\P\omega = \bigcup_{i=0}^{16g-14} \CylI\Q{(\eta_0,\dots,\eta_n,2\sigma(\eta_n/2)+4+i)}. \]
    
    \item If $\omega_n$ is even and either $\omega_{n-1}$ is odd or $(\omega_{n-1},\omega_n) = (2m,2\sigma(m)+4)$ for some $m$ (but not all $\omega_i$ are even), then $\eta_n$ is odd and
    \[ \qquad\qquad\CylI\P\omega = \CylI\Q{(\eta_0,\dots,\eta_n,\eta_{n+1})} \cup \CylI\Q{(\eta_0,\dots,\eta_n,\eta_{n+1}+1)} \cup  \bigcup_{i=0}^{16g-16} \CylI\Q{(\eta_0,\dots,\eta_n+1,2\sigma(\frac{\eta_n+1}{2})+6+i)}, \] 
    where $\eta_{n+1} = 2\sigma(\floor{\eta_n/2})+8g-5$.
\end{enumerate}

We begin with the base case $n=0$ for all parts. The original Markov partition sets $I_i$ are both $\P$- and $\Q$-cylinders:
\[ \CylI\P{(\omega_0)} = I_{\omega_0} = \CylI\Q{(\omega_0)}. \]
For $\omega_0 = 2k-1$ odd,
\begin{align*} \CylI\P{\omega} = \CylI\Q{(2k-1)} = \CylI\Q{(2k-1,2\sigma(k)+8g-5)} \cup \CylI\Q{(2k-1,2\sigma(k)+8g-4)}, \end{align*}
and for $\omega_0 = 2k$ even,
\begin{align*}
	\CylI\P{\omega} = \CylI\Q{(2k)}
	&= \CylI\Q{(2k,2\sigma(k)+4)} \cup \CylI\Q{(2k,2\sigma(k)+5)} 
	\cup \cdots \cup \CylI\Q{(2k,2\sigma(k)-5)} \cup \CylI\Q{(2k,2\sigma(k)-4)}
\end{align*}
by~\eqref{Markov odd Q} and~\eqref{Markov even} with $f_\A = f_\Q$.
	
For some parts of the proof, $n=0$ is a sufficient base case, but we do at times implicitly assume $n \ge 1$, so we also provide here a ``base case'' with $n=1$. If $\omega_0$ is even, then
\[ \CylI\P{(\omega_0,\omega_1)} = \CylI\Q{(\omega_0,\omega_1)}, \]
and equations (\ref{recode 2}) and (\ref{recode 17}) follow immediately when $\omega_1$ is odd, or, respectively, even.
If $\omega_0=2k-1$ is odd, then $\omega_1$ can be either $2\sigma(k)+2$ or $2\sigma(k)+3$. We investigate the interval $\CylI\P{(2k-1,2\sigma(k)+3)}$. 
For that, notice that
\(
    \CylI\P{(2k-1,2\sigma(k)+3)} = T^{-1}_k(I_{2\sigma(k)+3})
\).
From relation~\eqref{Markov odd Q} written for index $\sigma(k)+2$, the interval $I_{2\sigma(k)+3}$ itself can be expressed as
\begin{align*}
I_{2\sigma(k)+3}&=T^{-1}_{\sigma(k)+1}([P_{\sigma(\sigma(k)+1))-1},P_{\sigma(\sigma(k)+1))}])
=T^{-1}_{\sigma(k)+1}[P_{(k-2)+4g-2},P_{(k-2)+4g-1}] \\*
&=T^{-1}_{\sigma(k)+1}(I_{2(k-2)+8g-5}\cup I_{2(k-2)+8g-4)})
=T^{-1}_{\sigma(k)+1}(I_{2k+8g-9}\cup I_{2k+8g-8)}).
\end{align*}
Now we use \Cref{thm generator relation} 
to write
\begin{align*}
\CylI\P{(2k-1,2\sigma(k)+3)}&=T^{-1}_kT^{-1}_{\sigma(k)+1}(I_{2k+8g-9}\cup I_{2k+8g-8})\\
&=T^{-1}_{k-1}T^{-1}_{\sigma(k-1)-1}(I_{2k+8g-9}\cup I_{2k+8g-8}) \\
&=T^{-1}_{k-1}T^{-1}_{\sigma(k)+4g-2}(I_{2k+8g-9}\cup I_{2k+8g-8}) \\
&=T^{-1}_{k-1}\big(\CylI\Q{(2\sigma(k)+8g-4,2k+8g-9)}\cup \CylI\Q{(2\sigma(k)+8g-4,2k+8g-8)}\big) \\
&=\CylI\Q{(2k-1,2\sigma(k)+8g-4,2k+8g-9)}\cup \CylI\Q{(2k-1,2\sigma(k)+8g-4,2k+8g-8)}
\end{align*}
which proves~\eqref{recode 2}, that is, part (\ref{item recode 2}), for $n=1$.

The other $\P$-cylinder interval $\CylI\P{(2k-1,2\sigma(k)+2)}=\CylI\P{(2k-1)}\setminus \CylI\P{(2k-1,2\sigma(k)+3)}$. 
Since 
\begin{align*}
\CylI\P{(2k-1)}
=\CylI\Q{(2k-1)}
&=\CylI\Q{(2k-1,2\sigma(k)+8g-5)}\cup \CylI\Q{(2k-1,2\sigma(k)+8g-4)} \\
&=\CylI\Q{(2k-1,2\sigma(k)+8g-5,2k-1)}\cup \CylI\Q{(2k-1,2\sigma(k)+8g-5,2k)} \\* &\qquad \cup \bigcup_{\ell=2k+8g}^{2k+8g-8} \CylI\Q{(2k-1,2\sigma(k)+8g-4,\ell)}
\end{align*}
by \Cref{thm admissibility helper}(a) and, rewriting $\CylI\P{(2k-1,2\sigma(k)+3)}$ as the union of two $\Q$-cylinders above, we have
\begin{align*}
    \CylI\P{(2k-1,2\sigma(k)+2)}
    &= \CylI\Q{(2k-1,2\sigma(k)+8g-5,2k-1)} \cup \CylI\Q{(2k-1,2\sigma(k)+8g-5,2k)} \\* &\qquad \cup \bigcup_{\ell=2k+8g}^{2k+8g-10} \CylI\Q{(2k-1,2\sigma(k)+8g-4,\ell)},
\end{align*}
proving (\ref{recode 2+15}) for $n=1$.

\medskip
We proceed now with induction for $n \ge 2$. We say that a cylinder $\CylI\A{(\omega_0,\dots,\omega_n)}$ has \emph{rank} $n+1$. Assume $\CylI\P{\omega}$ of rank $\le n$ is a union of $\Q$-cylinders as desired; we want $\CylI\P{(\omega_0,\dots,\omega_n)}$ to be a union of $\Q$-cylinders of rank $n+2$.
	\smallskip
	
When $\omega_0$ is even, the induction argument is straightforward for both parts. We demonstrate it for part (\ref{item recode 2}), that is, when $\omega_n$ is odd.
Using the induction hypothesis for $\CylI\P{(\omega_1,\ldots,\omega_n)}$, we have
	\begin{align*} 
		\CylI\P{(\omega_0,\omega_1,\ldots,\omega_n)}
		&= T^{-1}_{\omega_0/2} \big( \CylI\P{(\omega_1,\ldots,\omega_n)} \big) \\*
		&= T^{-1}_{\omega_0/2} \big( \CylI\Q{(\eta_1,\ldots,\eta_n,\eta_{n+1})} \cup \CylI\Q{(\eta_1,\ldots,\eta_n,\eta_{n+1}+1)} \big) \text{ by induction} \\
		&= T^{-1}_{\omega_0/2} \CylI\Q{(\eta_1,\ldots,\eta_n,\eta_{n+1})} \cup T^{-1}_{\omega_0/2} \CylI\Q{(\eta_1,\ldots,\eta_n,\eta_{n+1}+1)} \\*
		&= \CylI\Q{(\omega_0,\eta_1,\ldots,\eta_n,\eta_{n+1})} \cup \CylI\Q{(\omega_0,\eta_1,\ldots,\eta_n,\eta_{n+1}+1)},
	\end{align*}
	where the final substitution uses the fact that $\eta_1 = \omega_1$ (from induction) and that the pair $(\omega_0,\eta_1) = (\omega_0,\omega_1)$ is $\P$-admissible if and only if it is $\Q$-admissible (because $\omega_0$ is even, and the even rows of $M_\P$ and $M_\Q$ are identical). Part (\ref{item recode 17}) can be treated similarly.

\medskip
We now prove parts (\ref{item recode 2}) and (\ref{item recode 17}) separately when $\omega_0$ is odd. 
\smallskip

\noindent\textbf{(\ref{item recode 2})}
From the induction hypothesis,
\[ \CylI\P{(\omega_1,\omega_2,\dots,\omega_n)}= \CylI\Q{(\xi_1,\xi_2,\dots,\xi_n,\xi_{n+1})} \cup \CylI\Q{(\xi_1,\xi_2,\dots,\xi_n,\xi_{n+1}+1)} \]
with $\xi_1 = \omega_1$ and $\xi_{n+1}$ odd. (We use $\xi$ here instead of $\eta$ because the terms $\xi_i$ will not necessarily be $\eta_i$ for $\CylI\P{(\omega_0,\dots,\omega_n)}$ from the statement of \Cref{thm P as unions of Q}.) Thus
\begin{align*}
    \CylI\P\omega
	&= T^{-1}_{\ceil{\omega_0/2}} \big(\CylI\P{(\omega_1,\omega_2,\omega_3,\dots,\omega_n)}\big) \\*
	&= T^{-1}_{\ceil{\omega_0/2}} \Big(\CylI\Q{(\omega_1,\xi_2,\xi_3,\dots,\xi_{n+1})} \cup \CylI\Q{(\omega_1,\xi_2,\xi_3,\dots,\xi_{n+1}+1)}\Big) \\
	&= T^{-1}_{\ceil{\omega_0/2}} T^{-1}_{\floor{\omega_1/2}} \Big(\CylI\Q{(\xi_2,\xi_3,\dots,\xi_{n+1})} \cup \CylI\Q{(\xi_2,\xi_3,\dots,\xi_{n+1}+1)}\Big).
\end{align*}
Let $\omega_0 = 2k-1$. Then $\omega_1$ must be $2\sigma(k)+2$ or $2\sigma(k)+3$, and either way $\floor{\omega_1/2} = \sigma(k)+1$, giving
\< \label{IPw first rewrite}
	\CylI\P\omega = T^{-1}_{k} T^{-1}_{\sigma(k)+1} \Big(\CylI\Q{(\xi_2,\xi_3,\dots,\xi_{n+1})} \cup \CylI\Q{(\xi_2,\xi_3,\dots,\xi_{n+1}+1)}\Big).
\>

There are now several cases and sub-cases to consider; these are summarized in \Cref{fig tree}.
\begin{figure}[htb] \vspace{-1em}
\includegraphics{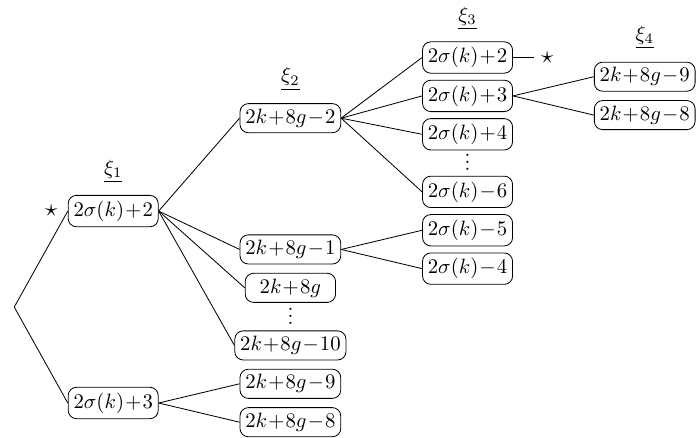}
\caption{Relevant cases when $\omega_0 = 2k-1$ is odd (all indices mod $16g-8$)}
\label{fig tree}
\end{figure}

If $\xi_1 = 2\sigma(k)+2$ then $\xi_2 \in \{2k+8g-2, 2k+8g-1, \dots, 2k+8g-10\}$, and if $\xi_1 = 2\sigma(k)+3$ then $\xi_2$ is $2k+8g-9$ or $2k+8g-8$. Other than when $\xi_2 \in \{ 2k+8g-2, 2k+8g-1 \}$, we can apply \Cref{thm generator relation} to~\eqref{IPw first rewrite} to get
\[ \label{eq:ipq}
	\CylI\P\omega = T^{-1}_{k-1} T^{-1}_{\sigma(k)+4g-2} \Big(\CylI\Q{(\xi_2,\xi_3,\dots,\xi_{n+1})} \cup \CylI\Q{(\xi_2,\xi_3,\dots,\xi_{n+1}+1)}\Big), \]
and then \Cref{thm admissibility helper}(\ref{item adm 3terms}) implies
\begin{align*}
    \CylI\P\omega
    &= T^{-1}_{\floor{\frac{2k-1}2}} T^{-1}_{\floor{\frac{2\sigma(k)+8g-4}2}} \Big(\CylI\Q{(\xi_2,\xi_3,\dots,\xi_{n+1})} \cup \CylI\Q{(\xi_2,\xi_3,\dots,\xi_{n+1}+1)}\Big) \\*
	&= \CylI\Q{(2k-1, 2\sigma(k)+8g-4, \xi_2,\xi_3,\dots,\xi_{n+1})} \cup \CylI\Q{(2k-1, 2\sigma(k)+8g-4, \xi_2,\xi_3,\dots,\xi_{n+1}+1)}.
\end{align*}
We are left with analyzing the cases $\xi_2=2k+8g-2$ and $\xi_2=2k+8g-1$, with $\xi_1 = 2\sigma(k)+2$. Here $\floor{\xi_2/2} = k+4g-1$ and so we proceed from~\eqref{IPw first rewrite} as
\begin{align*}
    \CylI\P\omega
	&= T^{-1}_k T^{-1}_{\sigma(k)+1} T^{-1}_{k+4g-1} \Big(\CylI\Q{(\xi_3,\dots,\xi_{n+1})} \cup \CylI\Q{(\xi_3,\dots,\xi_{n+1}+1)}\Big) \\*
	&= T^{-1}_{k-1} T^{-1}_{\sigma(k)+4g-3} T^{-1}_{k} \Big(\CylI\Q{(\xi_3,\dots,\xi_{n+1})} \cup \CylI\Q{(\xi_3,\dots,\xi_{n+1}+1)}\Big)
\end{align*}
using \Cref{thm generator inductive} with $m=1$.
If $\xi_2 = 2k+8g-2$, then $\xi_3 \in \{ 2\sigma(k)+2, 2\sigma(k)+3, \dots, 2\sigma(k)-6 \}$, and if $\xi_2 = 2k+8g-1$, then $\xi_3$ is $2\sigma(k)-5$ or $2\sigma(k)-4$. For all possible pairs $(\xi_2,\xi_3)$ \emph{except} $(\xi_2,\xi_3) = (2k+8g-2,2\sigma(k)+2)$ and $(\xi_2,\xi_3) = (2k+8g-2,2\sigma(k)+3)$, \Cref{thm admissibility helper}(\ref{item adm 4terms}) implies precisely that 
\begin{align*}
    \CylI\P\omega
	&= T^{-1}_{k-1} T^{-1}_{\sigma(k)+4g-3} T^{-1}_{k} \Big(\CylI\Q{(\xi_3,\dots,\xi_{n+1})} \cup \CylI\Q{(\xi_3,\dots,\xi_{n+1}+1)}\Big) \\*
	&= T^{-1}_{\floor{\frac{2k-1}2}} T^{-1}_{\floor{\frac{2\sigma(k)+8g-5}2}} T^{-1}_{\floor{\frac{2k}2}} \Big(\CylI\Q{(\xi_3,\dots,\xi_{n+1})} \cup \CylI\Q{(\xi_3,\dots,\xi_{n+1}+1)}\Big) \\*
	&= \CylI\Q{(2k-1,2\sigma(k)+8g-5,2k,\xi_3,\dots,\xi_{n+1})} \cup \CylI\Q{(2k-1,\dots,\xi_{n+1}+1)}.
\end{align*}

\smallskip
Now we only need to analyze the cases $\xi_3 = 2\sigma(k)+2$ and $\xi_3 = 2\sigma(k)+3$, where we have already set $\xi_2=2k+8g-2$ and $\xi_1 = 2\sigma(k)+2$. 

If $\xi_3=2\sigma(k)+3$, then $\xi_4$ is either $2k+8g-8$ or $2k+8g-9$, and so
\begin{align*}
    \CylI\P\omega
	&= T^{-1}_{k-1} T^{-1}_{\sigma(k)+4g-3} T^{-1}_{k} \Big(\CylI\Q{(\xi_3,\dots,\xi_{n+1})} \cup \CylI\Q{(\xi_3,\dots,\xi_{n+1}+1)}\Big) \\*
	&= T^{-1}_{k-1} T^{-1}_{\sigma(k)+4g-3} T^{-1}_{k} T^{-1}_{\sigma(k)+1} \Big(\CylI\Q{(\xi_4,\dots,\xi_{n+1})} \cup \CylI\Q{(\xi_4,\dots,\xi_{n+1}+1)}\Big) \\
	&= T^{-1}_{k-1} T^{-1}_{\sigma(k)+4g-3} T^{-1}_{k-1} T^{-1}_{\sigma(k)+4g-2} \Big(\CylI\Q{(\xi_4,\dots,\xi_{n+1})} \cup \CylI\Q{(\xi_4,\dots,\xi_{n+1}+1)}\Big) \quad\text{by \Cref{thm generator relation}} \\
	&= T^{-1}_{\floor{\frac{2k-1}2}} T^{-1}_{\floor{\frac{2\sigma(k)+8g-5}2}} T^{-1}_{\floor{\frac{2k-1}2}} T^{-1}_{\floor{\frac{2\sigma(k)+8g-4}2}} \Big(\CylI\Q{(\xi_4,\dots,\xi_{n+1})} \cup \CylI\Q{(\xi_4,\dots,\xi_{n+1}+1)}\Big) \\*
	&= \CylI\Q{(2k-1,2\sigma(k)+8g-5,2k-1,2\sigma(k)+8g-4,\xi_4,\dots,\xi_{n+1})} \cup \CylI\Q{(2k-1,\dots,\xi_{n+1}+1)}
\end{align*}
by \Cref{thm admissibility helper}(\ref{item adm 5terms}).

If $\xi_3=2\sigma(k)+2$, notice that $\xi_3 = \xi_1$, so we now analyze the situation when the sequence $(\xi_1,\dots,\xi_{n+1})$ consists of several alternating entries $(2\sigma(k)+2,2k+8g-2)$ until some $\xi_j \notin \{ 2\sigma(k)+2, 2k+8g-2 \}$ (this situation is denoted by~{\large$\star$} in \Cref{fig tree}). Notice that $j<n+1$: otherwise, all $\xi_1,\dots,\xi_{n}$ would be even, and then
\[
	\CylI\Q{(\xi_1,\xi_2,\dots,\xi_{n-1},\xi_n,\xi_{n+1})} 
	\subset \CylI\Q{(\xi_1,\xi_2,\dots,\xi_{n-1},\xi_n)}
	= \CylI\P{(\xi_1,\xi_2,\dots,\xi_{n-1},\xi_n)}
\]
would imply $\CylI\P{(\xi_1,\xi_2,\dots,\xi_n)}=\CylI\P{(\omega_1,\omega_2,\dots,\omega_n)}$, which is not possible since $\omega_n$ is odd.

We assume $j$ is odd (the case of even $j$ can be treated similarly). Then
\[ (\xi_1,\xi_2,\dots,\xi_{n+1})=(2\sigma(k)+2,2k+8g-2,\dots,2\sigma(k)+2,2k+8g-2,\xi_j,\dots,\xi_{n+1}), \]
where $\xi_j$ is one of $\{2\sigma(k)+3, \dots, 2\sigma(k)-6\}$. Thus
\begin{align}
    \CylI\P\omega
	&= T^{-1}_k \Big(\CylI\Q{(\xi_1,\xi_2,\xi_3,\dots,\xi_{n+1})} \cup \CylI\Q{(\xi_1,\dots,\xi_{n+1}+1)}\Big) \nonumber \\*
	&= T^{-1}_k \big( T^{-1}_{\sigma(k)+1} T^{-1}_{k+4g-1} \big)^{(j-1)/2} \Big(\CylI\Q{(\xi_j,\dots,\xi_{n+1})} \cup \CylI\Q{(\xi_j,\dots,\xi_{n+1}+1)}\Big) \nonumber \\*
	\label{periodic rewritten}
	&= \big( T^{-1}_{k-1} T^{-1}_{\sigma(k)+4g-3} \big)^{(j-1)/2} T^{-1}_k \Big(\CylI\Q{(\xi_j,\dots,\xi_{n+1})} \cup \CylI\Q{(\xi_j,\dots,\xi_{n+1}+1)}\Big)
\end{align}
by \Cref{thm generator inductive}. For $\xi_j \ne 2\sigma(k)+3$, \Cref{thm admissibility helper}(\ref{item adm 4terms}) implies that 
\begin{align*}
    \CylI\P\omega
	&= \big( T^{-1}_{\floor{\frac{2k-1}2}} T^{-1}_{\floor{\frac{2\sigma(k)+8g-5}2}} \big)^{(j-1)/2} T^{-1}_{\floor{\frac{2k}2}} \Big(\CylI\Q{(\xi_j,\dots,\xi_{n+1})} \cup \CylI\Q{(\xi_j,\dots,\xi_{n+1}+1)}\Big) \\*
	&= \CylI\Q{(2k-1, 2\sigma(k)+8g-5, \dots, 2k-1, 2\sigma(k)+8g-5, 2k, \xi_j,\dots,\xi_{n+1})} \cup \CylI\Q{(2k-1,\dots,\xi_{n+1}+1)},
\end{align*}
and for $\xi_j = 2\sigma(k)+3$ we proceed from~\eqref{periodic rewritten} with
\begin{align*}
    \CylI\P\omega
	&= \big( T^{-1}_{k-1} T^{-1}_{\sigma(k)+4g-3} \big)^{(j-1)/2} T^{-1}_k \Big(\CylI\Q{(\xi_j,\dots,\xi_{n+1})} \cup \CylI\Q{(\xi_j,\dots,\xi_{n+1}+1)}\Big) \\
	&= \big( T^{-1}_{k-1} T^{-1}_{\sigma(k)+4g-3} \big)^{(j-1)/2} T^{-1}_k T^{-1}_{\sigma(k)+1} \Big(\CylI\Q{(\xi_{j+1},\dots,\xi_{n+1})} \cup \CylI\Q{(\xi_j,\dots,\xi_{n+1}+1)}\Big) \\
	&= \big( T^{-1}_{k-1} T^{-1}_{\sigma(k)+4g-3} \big)^{(j-1)/2} T^{-1}_{k-1} T^{-1}_{\sigma(k)+4g-2} \Big(\CylI\Q{(\xi_{j+1},\dots,\xi_{n+1})} \cup \CylI\Q{(\xi_j,\dots,\xi_{n+1}+1)}\Big) \\
	&= \CylI\Q{(2k-1,2\sigma(k)+8g-5, \dots, 2k-1, 2\sigma(k)+8g-4, \xi_{j+1},\dots,\xi_{n+1})} \cup \CylI\Q{(2k-1,\dots,\xi_{n+1}+1)},
\end{align*}
using \Cref{thm generator relation} for a substitution and then \Cref{thm admissibility helper}(\ref{item adm 3terms}) for the final line since, following $\xi_j = 2\sigma(k)+3$, we know $\xi_{j+1}$ is either $2k-8g-9$ or $2k-8g-8$. Having followed all paths in \Cref{fig tree}, this completes the proof of part~(\ref{item recode 2}).

\medskip\noindent\textbf{(\ref{item recode 17})} When $\omega_n$ is even, there are two possible structures,~\eqref{recode 17} and~\eqref{recode 2+15}, for the decomposition of $\CylI\P\omega$, corresponding to the two cases (a) and (b) on page~\pageref{a-and-b}.

\smallskip\noindent\textbf{(a)} First, if all $\omega_i$ are even then $\CylI\P\omega = \CylI\Q\omega$ because the even rows of $M_\P$ and $M_\Q$ coincide. Then $\CylI\P\omega$ can be trivially decomposed into $16g-15$ cylinders of higher rank as in~\eqref{recode 17}.

If $(\omega_{n-1},\omega_n)$ are both even and not of the form $(2m,2\sigma(m)+4)$ for any $m$, then, from the induction hypothesis for case~(a), there exists a $\Q$-admissible sequence $(\xi_1,\xi_2,\dots,\xi_{n})$ such that
\[ \CylI\P{(\omega_1,\omega_2,\dots,\omega_n)} = \CylI\Q{(\xi_1,\xi_2,\dots,\xi_{n})}, \]
with $\xi_n$ even. Now the analogue of relation~\eqref{IPw first rewrite} is
\[ \CylI\P\omega = T^{-1}_{k} T^{-1}_{\sigma(k)+1} \big(\CylI\Q{(\xi_2,\dots,\xi_{n})}\big). \]
The sequence $(\xi_1,\xi_2,\dots,\xi_{n})$  cannot consist entirely of alternating even entries $(2\sigma(k)+2,2k+8g-2)$: if this were the case, then \[ \CylI\Q{(\xi_1,\xi_2,\dots,\xi_{n})}=\CylI\P{(\xi_1,\xi_2,\dots,\xi_{n})}, \text{ so } (\omega_1,\omega_2,\dots,\omega_n)= (\xi_1,\xi_2,\dots,\xi_{n}),\] 
which is impossible since the last two entries $(\omega_{n-1},\omega_n)$ are not of the form $(2m,2\sigma(m)+4).$ One can then proceed as in case (\ref{item recode 2}) and express $\CylI\P\omega$ as a single $\Q$-cylinder of rank $n+1$, which is then a union of $16g-15$ $\Q$-cylinders of rank $n+2$ as desired.

\noindent\textbf{(b)} If $\omega_{n-1}$ is odd or the final pair $(\omega_{n-1},\omega_n) = (2m,2\sigma(m)+4)$ for some $m$, then we have~\eqref{recode 2+15}, as will we now show.

We follow the proof of (\ref{item recode 2}), where a stricter key step $j<n$ will now follow from the new assumptions. Indeed, from the induction hypothesis for case~(b),
\[ 
    \CylI\P{(\omega_1,\omega_2,\dots,\omega_n)} = \CylI\Q{(\xi_1,\dots,\xi_n,\xi_{n+1})} \cup \CylI\Q{(\xi_1,\dots,\xi_n,\xi_{n+1}+1)} \cup \!\bigcup_{i=0}^{16g-16}\! \CylI\Q{(\xi_1,\dots,\xi_n\!+\!1,2\sigma(\frac{\xi_n+1}{2})\!+\!6\!+\!i)},
\] 
where $\xi_1=\omega_1$, $\xi_n$ is odd, and $\xi_{n+1} = 2\sigma(\floor{\xi_n/2})+8g-5$. The analogous statement to~\eqref{IPw first rewrite} is now
\[
    \CylI\P\omega = T^{-1}_{k} T^{-1}_{\sigma(k)+1} \Big( \CylI\Q{(\xi_2,\dots,\xi_{n+1})} \cup \CylI\Q{(\xi_2,\dots,\xi_{n+1}+1)} \cup \!\bigcup_{i=0}^{16g-16}\! \CylI\Q{(\xi_2,\dots,\xi_n\!+\!1,2\sigma(\frac{\xi_n+1}{2})\!+\!6\!+\!i)}\Big).
\]
Notice that it is not possible for the sequence $(\xi_1,\xi_2,\dots,\xi_n)$ to consist  entirely of alternating entries $(2\sigma(k)+2,2k+8g-2)$ because $\xi_n$ is odd. Nor can the sequence $(\xi_1,\xi_2,\dots,\xi_n+1)$ consist entirely of such alternating even entries because then $(\xi_{n-1},\xi_n)$ would not be $\Q$-admissible.

Therefore,  there exists $j<n$ such that the sequence $(\xi_1,\dots,\xi_j)$ stops alternating between $2\sigma(k)+2$ and $2k+8g-2$. We can then express each $T^{-1}_{k} T^{-1}_{\sigma(k)+1} \big(\CylI\Q{(\xi_2,\dots)}\big)$ above as a $\Q$-cylinder of rank $n+2$, thus making $\CylI\P\omega$ a union of $2+(16g-17)=16g-15$ $\Q$-cylinders of rank $n+2$. This does not affect the last two entries of the $\Q$-cylinders from the induction hypothesis, so the structure of the decomposition is as needed.


\begin{thebibliography}{99}
\bibitem{A20} A.~Abrams. Extremal parameters and their duals for boundary maps associated to Fuchsian groups, \textit{Illinois J. Math.} \textbf{65} (2021) No.~1, 153--179. 

\bibitem{AK19} A.~Abrams, S.~Katok. Adler and Flatto revisited: cross-sections for geodesic~flow on compact surfaces of constant negative curvature. \textit{Studia Mathematica} \textbf{246} (2019), 167--202.

\bibitem{AKU-Flexibility} A.~Abrams, S.~Katok, I.~Ugarcovici. Flexibility of measure-theoretic entropy of boundary maps associated to Fuchsian groups. \textit{Ergodic Theory \& Dynamical Systems} \textbf{42} (2022), 389--401.

\bibitem{AF91} R.~Adler, L.~Flatto. Geodesic flows, interval maps, and symbolic dynamics.  \textit{Bull. Amer. Math. Soc.} \textbf{25} (1991), No.~2, 229--334.

\bibitem{AKM} R.~Adler, A.~Konheim, M.~McAndrew. Topological entropy. \textit{Trans. Amer. Math. Soc.} \textbf{114} (1965), 309--319.

\bibitem{ALM} L.~Alsed{\`a}, J.~Llibre, M.~Misiurewicz. Combinatorial Dynamics and Entropy in Dimension One, Second Edition, Advanced Series in \textit{Nonlinear Dynamics 5}, World Scientific Publishing Co. Inc., River Edge, NJ, 2000.

\bibitem{AM} L.~Alsed{\`a}, M.~Misiurewicz. Semiconjugacy to a map of a constant slope. \textit{Discrete {\sl\&} Continuous Dynamical Systems B}. \textbf{20} (10) 2015, 3403--3413.

\bibitem{B71} R.~Bowen. Entropy for group endomorphisms and homogeneous spaces. \textit{Trans. American Math. Society.} \textbf{153} (1971), 401--414, erratum, \textbf{181} (1973), 509--510.

\bibitem{BS79} R.~Bowen, C.~Series. Markov maps associated with Fuchsian groups. \textit{Inst. Hautes \'Etudes Sci. Publ. Math.} No. 50 (1979), 153--170.

\bibitem{Denker} M.~Denker, G.~Keller, M.~Urba\'nski. On the uniqueness of equilibrium states for piecewise monotone mappings. \textit{Studia Math.} \textbf{97} (1990), 27--36.

\bibitem{D70} E.~Dinaburg. The relation between topological entropy and metric entropy. \textit{Soviet Math. Dokl.} \textbf{11} (1970), 13--16.

\bibitem{Hofbauer} F.~Hofbauer. On intrinsic ergodicity of piecewise monotonic transformations with positive entropy II, \textit{Israel J. Math.} \textbf{38} (1981), 107--115.

\bibitem{G41} I.~Gelfand. Normierte ringe (German). \textit{Rec. Math. [Mat. Sbornik] N.S.}, \textbf{9(51)}:1 (1941), 3--24.


\bibitem{K92} S.~Katok. \textit{Fuchsian Groups.} University of Chicago Press, 1992.

\bibitem{KU17} S.~Katok, I.~Ugarcovici. Structure of attractors for boundary maps associated to Fuchsian groups. \textit{Geometriae Dedicata} \textbf{191} (2017), 171--198.

\bibitem{KU17e} S.~Katok, I.~Ugarcovici. Correction to: Structure of attractors for boundary maps associated to Fuchsian groups. \textit{Geometriae Dedicata} \textbf{198} (2019), 189--181.

\bibitem{MTh} J.~Milnor, W.~Thurston. On iterated maps of the interval, in \textit{Dynamical Systems} College Park, MD, 1986 (87), Lecture Notes in Math., 1342, Springer, Berlin, 1988, 465--563.

\bibitem{MSz} M.~Misiurewicz, W.~Szlenk. Entropy of piecewise monotone mappings. \textit{Studia Mathematica} \textbf{67} (1980), 45--63.

\bibitem{MZ} M.~Misiurewicz, K.~Ziemian. Horseshoes and entropy for piecewise continuous piecewise monotone maps, in \textit{From Phase Transitions to Chaos}, World Sci. Publ., River Edge, NJ, 1992, 489--500.

\bibitem{Parry64} W.~Parry. Intrinsic Markov chains. \textit{Trans. Amer. Math. Soc.} \textbf{112} (1964), 55--55. 

\bibitem{Parry66} W.~Parry. Symbolic dynamics and transformations of the unit interval. \textit{Trans. Amer. Math. Soc.} \textbf{122} (1966), 368--378. 



\bibitem{W} P.~Walters, \textit{An Introduction to Ergodic Theory}, Graduate Texts in Mathematics, {\bf 79}, Springer-Verlag, 1975.

\end{thebibliography}
\end{document}